\documentclass[11pt,a4paper]{amsart}
\usepackage[latin1]{inputenc}
\usepackage{amsmath}
\usepackage{amsfonts}
\usepackage{amssymb}
\usepackage{amsthm}
\usepackage{xfrac}
\usepackage{graphicx}
\usepackage{enumitem}
\usepackage{nicefrac}

\usepackage{xcolor}

\usepackage{hyperref}
\definecolor{darkblue}{rgb}{0.0,0.0,0.3}
\hypersetup{colorlinks,breaklinks,
            linkcolor=darkblue,urlcolor=darkblue,
            anchorcolor=darkblue,citecolor=darkblue}

\usepackage{afterpage}

\theoremstyle{plain}
\newtheorem{theorem}{Theorem}[section]
\newtheorem{prop}[theorem]{Proposition}
\newtheorem*{theorem*}{Theorem}
\newtheorem*{prop*}{Proposition}

\newtheorem{lemma}[theorem]{Lemma}
\newtheorem{definition}[theorem]{Definition}

\newtheorem*{remark*}{Remark}
\newtheorem{corollary}[theorem]{Corollary}
\newtheorem*{corollary*}{Corollary}

\newtheoremstyle{named}{}{}{\itshape}{1.5em}{\scshape}{.}{5pt plus 1pt minus 1pt}{\thmnote{#3}}
\theoremstyle{named}

\numberwithin{equation}{section}
\numberwithin{figure}{section}

\author{Siegfried Beckus, Latif Eliaz}%$^{1}$}
\title[Growth of Eigenfunctions on Graphs]{Growth of Eigenfunctions and $\mathcal{R}$-limits on Graphs}

\address{Institut f\"ur Mathematik\\
Universit\"at Potsdam\\
Potsdam, Germany}
\email{beckus@uni-potsdam.de}

\address{Department of Mathematics\\ 
Technion Israel Institute of Technology\\ 
Haifa, Israel}
\email{latif@campus.technion.ac.il}

\begin{document}
\sloppy
%\thanks{$^1$
%Supported in part by the Israel Science Foundation (Grant No.\ 399/16).}
%\sloppy

\setcounter{tocdepth}{2}

\begin{abstract}
A characterization of the essential spectrum $\sigma_{\text ess}$ of Schr\"odinger operators on infinite graphs is derived involving the concept of $\mathcal{R}$-limits. This concept, which was introduced previously for operators on $\mathbb{N}$ and $\mathbb{Z}^d$ as ``right-limits", captures the behaviour of the operator at infinity. For graphs with sub-exponential growth rate we show that each point in $\sigma_{\text ess}(H)$ corresponds to a bounded generalized eigenfunction of a corresponding $\mathcal{R}$-limit of $H$. If, additionally, the graph is of uniform sub-exponential growth, also the converse inclusion holds.
\end{abstract}

\maketitle

\section{Introduction}
\label{sec:Intro}

This work deals with Schr\"odinger operators $H:\ell^2(G)\to\ell^2(G)$ of the form
\begin{equation}\label{eq:H}
(H\psi)(v) 
	:= \sum_{u\sim v} \big(\psi(u)-\psi(v)\big) + W(v)\psi(v)
\end{equation}
where $W:V(G)\to\mathbb{R}$ is a bounded function (the potential) and $G$ is an infinite, connected graph with a uniform bound on the vertex degree.

\medskip

Weyl's theorem asserts that the essential spectrum of a linear bounded self-adjoint operator is invariant under compact perturbations. In light of this, one naturally expects that the essential spectrum only depends on the geometry at infinity of the underlying space. This relation has been exposed for Schr\"odinger operators on $\mathbb{N}$ or $\mathbb{Z}^n$. More precisely, the essential spectrum was characterized by the union of the sets $\sigma_\infty(H')$ where $H'$ runs over all right-limits of $H$ and $\sigma_\infty(H')$ denotes the set of bounded generalized eigenfunctions, see \cite{CWL11,SimSz11}. 
Here the term right limit %is used in order to reflects
refers to the study of both the potential and the geometry at infinity. 
In particular, if $G$ is the Cayley graph of $\mathbb{Z}$ with the usual generators, a right limit is a strong limit point of a sequence of shifts going to infinity of the original operator. This notion of right limits was extended to $\mathbb{Z}^n$ in \cite{LSess06} and recently to general graphs with uniformly bounded vertex degree in \cite{BDE18}. Since the name right limit is no longer appropriate, these operators are called $\mathcal{R}$-limits. 

\medskip

For infinite connected graphs with uniform bound on the vertex degree, the authors of \cite{BDE18} proved that the union over the spectra of all $\mathcal{R}$-limits of $H$ is contained in the essential spectrum of $H$, where $H$ is a bounded Jacobi operator. Moreover they show that the converse inclusion holds on regular trees, and that on the contrary there exists an infinite, connected graph with uniform bound on the vertex degree such that equality does not hold. This is the starting point of the present work. We show that % the inclusion%\textcolor{red}{spheres} V4.1
%in the graph, the converse inclusion is obtained, namely
$$
\sigma_{\text{ess}}\left(H\right)
	\subseteq\bigcup_{H'\in\mathcal{R}(H)}\sigma_{\infty}\left(H'\right)
$$
holds under additional assumptions on the growth rate of balls, see Theorem~\ref{thm:sigmainfty}. In general this inclusion can be strict, see Proposition~\ref{prop:strictinclusion}.
However, if additionally the graph admits a uniform sub-exponential growth, then this inclusion is an equality and moreover, 
\begin{align}
\sigma_{\text ess}(H)
	=\bigcup_{H'\in\mathcal{R}(H)}\sigma(H')
	=\bigcup_{H'\in\mathcal{R}(H)}\sigma_\infty(H'). \nonumber
\end{align}
see Theorem~\ref{thm:fullresultsubexp}.

\medskip

Such results go back to the concept of limit operators based on Favard \cite{Fav14}, Muhamadiev \cite{Muh71,Muh72}, Lange-Rabinovich \cite{LR86}, Rabinovich-Roch-Silbermann \cite{RRS98,RRSband01,RRS04} and Chandler-Wilde-Lindner \cite{CWL08,CWL11}, see also \cite{Lin06,LS14,RRR04,SW17}. % It is also worth pointing out that in a more recent development \cite{HaSe19} these methods were   %tremendously V4.1 pushed forward to metric measure spaces.
In recent developments \cite{HaSe19} pushed these methods forward to metric measure spaces including graphs of property A \cite{Yu00}. Besides other things it is shown there that the essential spectrum coincides with the union of the spectra of so called limit operators.
Another approach to tackle such questions comes from $C^\ast$-algebras \cite{Geo11,Geo18,GI01,GI02,GI04,Man02} which uses the concept of localization at infinity, which coincides with the concept of $\mathcal{R}$-limits. % and leads to connections with coarse geometry. V4.1
In this case similar results for operators on locally compact, non-compact abelian groups are obtained. This was recently extended to groupoid $C^\ast$-algebra \cite{AZ19,CNQ17,CNQ18}.
For a more comprehensive review and further references on the subject see \cite{CWL11,LSess06,SimSz11}. 
Another recent work \cite{GolTru20} develops a similar characterization for the essential spectrum of the Laplacian on {\it Klaus-sparse graphs}. Note that the class of {\it Klaus-sparse graphs} have a non-trivial intersection with the class of uniform sub-exponential growing graph, and neither of these classes contains the other.
Further, \cite{ADY20} takes advantage of the above mentioned result on trees from \cite{BDE18}, to calculate the essential spectrum of Jacobi matrices on homogeneous trees, which are generated by an {\it Angelesco system}. 
\medskip

In order to prove our second main result (Theorem~\ref{thm:fullresultsubexp}), the so-called Shnol-type theorem is used. Shnol \cite{Shnol} proved that if a generalized eigenfunction admits at most a polynomial growth rate then the corresponding energy is in the spectrum of the operator. 
This result was independently discovered by Simon \cite{SimEF81}. Since then various remarkable generalizations to the Dirichlet setting were proven, see e.g.\ \cite{BP18,BD19,BdmLS09} and references therein. In the literature also the converse question is addressed \cite{BdmSt03,FLW14,LT16}. To be more precise, one seeks to find for $\mu$-almost every element in the spectrum, a generalized eigenfunction that has at most sub-exponential %(polynomial) V4.1 
growth, where $\mu$ is the spectral measure of the operator $H$. 
Such a converse theorem is used in the proof of Theorem~\ref{thm:sigmainfty}. %V4.1

\subsection{Organisation}
The main results of this work are presented in Section~\ref{sec:MainRes}. In Section~\ref{sec:Exam}, two examples are provided where the essential spectrum can be computed with the help of the main results. 
After introducing the main concepts such as $\mathcal{R}$-limits, the proof of the main Theorem~\ref{thm:sigmainfty} as well as of Proposition~\ref{prop:strictinclusion} is provided in Section~\ref{sec:CharEssSpec}. Then the proof of Theorem~\ref{thm:fullresultsubexp}, which states the equality, is given in Section~\ref{sec:Thm-fullresultsubexp}.

\subsection{Acknowledgements} %V4.1
Parts of this work are included in the PhD thesis of L.~Eliaz~\cite{E}\footnote{Submitted: February 2019.}, carried under the supervision of J.~Breuer from the Hebrew University of Jerusalem. We are grateful to him for his significant support during the preparation of this work. The authors wish to thank M.~Keller for inspiring discussions and for pointing out the reference \cite{Brooks}. L.~Eliaz acknowledges the support of the Israel Science Foundation (grants No.\ 399/16 and 970/15) founded by the Israel Academy of Sciences and Humanities. S.~Beckus is thankful for financial support of the Israel Science Foundation (grant No.\ 970/15) founded by the Israel Academy of Sciences and Humanities during the Postdoctoral period at the Israel Institute of Technology where the main part of this work was established.

\section{Setting and main results}
\label{sec:MainRes}

%\textcolor{red}{Check if concept of $d$-graphs was already used in the literature} V4.1

A graph $G$ consists of a countable vertex set $V(G)$ and an edge set $E(G)$ where an edge is represented by a tuple of vertices. Throughout this work we deal with undirected graphs and so the edge $(u,v)$ is identified with the edge $(v,u)$ for $u,v\in V(G)$. The tuple $(u,u)$ for some $u\in V(G)$ is called a {\em loop}. We only consider graphs without loops. Two vertices $u$ and $v$ are called {\em adjacent} ($u\sim v$) if $(u,v)\in E(G)$. Then the {\em vertex degree} $deg(v)$ of a vertex $v\in V(G)$ is defined by 
$$
deg(v):=\sharp\{u\in V(G)\,:\, v\sim u\},
$$ 
%\textcolor{red}{do we need to count loops twice?? Check the literature}
where $\sharp S$ denotes the cardinality of the set $S$. The tuple $(G,v_0)$ is called {\em rooted $d$-bounded graph} if $v_0\in V(G)$ is a fixed vertex and $deg(v)\leq d$ for all $v\in V(G)$. 
A {\em path} between two vertices $u,v\in G$ is given by a chain of vertices $(u_1,\ldots, u_n)$ satisfying $u_1=u, u_n=v$ and $u_i\sim u_{i+1}$ for all $1\leq i\leq n-1$. Then a graph is called {\em connected} if there is a path between any two vertices $u,v\in V(G)$.

\medskip
Let $\ell^2(G):=\ell^2(V(G))$ denote the Hilbert space of all square summable functions $\psi:V(G)\to\mathbb{C}$. Furthermore, $\ell^\infty(G):=\ell^\infty(V(G))$ is the Banach space of bounded functions $\psi:V(G)\to\mathbb{C}$ equipped with the norm $\|\psi\|_\infty:=\sup_{u\in V(G)}|\psi(u)|$.

\medskip

Throughout this work, we study the self-adjoint, linear and bounded operators acting on the Hilbert space $\ell^2(G)$ of the form \eqref{eq:H}, where $G$ is an infinite (i.e.~$\sharp V(G)=\infty$), $d$-bounded and connected graph.
Whenever $W$ is chosen to be identically zero, the operator is denoted by $\Delta$, which is called the {\em graph Laplacian}. Furthermore, $A=A_G$ denotes the {\em adjacency operator} on the graph $G$, which is a Schr\"odinger operator with $W(v):=deg(v)$ for $v\in V(G)$. A triple $(H,G,v_0)$ denotes a Schr\"odinger operator of the form \eqref{eq:H} defined on the rooted graph $(G,v_0)$.

\medskip

Let $H$ be a Schr\"odinger operator on a rooted graph $(G,v_0)$. The {\em spectrum} of $H$ is denoted by $\sigma(H)$. Then the {\em discrete spectrum} $\sigma_{disc}(H)\subseteq\sigma(H)$ is the set of isolated eigenvalues of finite multiplicity and the {\em essential spectrum} is $\sigma_{ess}(H):=\sigma(H)\setminus\sigma_{disc}(H)$.
Furthermore, a function $\psi:V(G)\to\mathbb{C}$ is called a {\em generalized eigenfunction of $H$ with eigenvalue $\lambda$} if $\psi\not\equiv 0$ and $H\psi(v) = \lambda \psi(v)$ for all $v\in V(G)$. With this at hand, $\sigma_\infty(H)$ denotes the set of all $\lambda$ such that there exists a bounded generalized eigenfunction $\psi\in\ell^\infty(G)$ with eigenvalue $\lambda$.

\medskip

The {\em combinatorial graph distance} on $G$ is defined by
$$
\text{dist}\left(u,v\right) := \inf\big\{ n \,\big|\, (v_0,v_1,\ldots,v_n) \text{ is a path with } v_0=u \text{ and } v_n=v \big\}.
$$
%Note that this expression gives the smallest number of edges which are contained in a path connecting $u$ and $v$. 
For a rooted graph $(G,v_0)$, the notation
\[
\left|v\right|:=\text{dist}\left(v,v_{0}\right)
\]
is used for the distance of a vertex $v\in V(G)$ from the root $v_0$. Then $\mathcal{S}_r(v_0)$ denotes the {\em sphere} of radius $r\in\mathbb{N}$ around the root $v_0$ and $B_r(v_0)$ is the {\em ball} of radius $r\in\mathbb{N}$ around the root $v_0$, namely,
$$
\mathcal{S}_r(v_0)
	:=\left\{v\in G\,|\, \text{dist}(v,v_0)=r\right\}, \qquad
B_r(v_0)
	:=\left\{v\in G\,|\, \text{dist}(v,v_0)\leq r\right\}.
$$
%%%%%%%%%%%%%%%%
\begin{definition}
\label{def:subexp}
A connected rooted graph $(G,v_0)$ is of {\em sub-exponential growth rate} if for each $\gamma>1$, there exists $C=C_{\gamma,v_0}>0$ such that for every $r\in\mathbb{N}$,
\begin{equation*}
\sharp B_r(v_0)< C\gamma^r.
\end{equation*}
Furthermore, a graph $G$ is of {\em uniform sub-exponential growth rate} if the constant $C>0$ can be chosen independently of the choice of the root. Specifically, for each $\gamma>1$, there exists a constant $C=C_{\gamma}>0$ such that
\begin{equation*}
\sharp B_r(u)< C\gamma^r
\end{equation*}
holds for every $u\in G$ and $r\in\mathbb{N}$. 
\end{definition}
%%%%%%%%%%%%%%%%

\medskip

The concept of $\mathcal{R}$-limits of a Schr\"odinger operator $H$ defined on a graph $G$ was recently introduced in \cite{BDE18}. A precise mathematical definition is provided in Section~\ref{ssect:R-limits}.

%%%%%%%%%%%%%%%%
\begin{theorem}
\label{thm:sigmainfty} 
Let $\left(G,v_0\right)$ be an infinite and connected $d$-bounded graph of sub-exponential growth rate, and $H$ be a Schr\"odinger operator on
$\ell^{2}\left(G\right)$ of the form \eqref{eq:H}. Then
\begin{equation*}
%\label{eq:essubsetinfty}
\sigma_{\text{ess}}\left(H\right)
	\subseteq\bigcup_{H'\in\mathcal{R}(H)}\sigma_{\infty}\left(H'\right).
\end{equation*} 
%\bigcup_{L\text{ is an \ensuremath{\mathcal{R}}-limit of \ensuremath{H}}}\sigma\left(L\right)\subseteq
\end{theorem}
%%%%%%%%%%%%%%%%

We point out that in the latter assertion it is not assumed that the graph is of uniform sub-exponential growth. The inclusion is preserved also for the adjacency operator $H:=A_G$ on the $d$-regular tree $G:=T_d$, although this graph has exponential growth rate. 
In this case one can directly check that this inclusion is strict. Indeed, the only $\mathcal{R}$-limit is the same operator $H'=A_{T_d}$ on $T_d$. Then $\sigma_{\text{ess}}(H)=\sigma(H)=\left[-2\sqrt{d-1},2\sqrt{d-1}\right]$ and $[-d,d]\subset\sigma_{\infty}(H)$ holds (see e.g.\ \cite[Theorem~1.1]{Brooks} and \cite{FP83}).
%, see e.g. \cite[Secion 7.c]{MW89} and references therein.
%Since also $[-d,d]\subset\sigma_{\infty}(H)$  holds by \cite[Equations~(7.6--7.7)]{MW89}, \green{They don't treat the case $p=\infty$, so from where did you get $[-d,d]\subset\sigma_{\infty}(H)$?}
Thus, we derive
\[
\left[-d,d\right]\big\backslash\left[-2\sqrt{d-1},2\sqrt{d-1}\right]\subset\sigma_{\infty}(H)\backslash\sigma_{\text{ess}}(H)
\] %\neq\emptyset 
where the set on the left hand side is non-empty. With this idea at hand, we also prove the following.

%%%%%%%%%%%%%%%%
\begin{prop}
\label{prop:strictinclusion}
There exists an infinite and connected $d$-bounded graph $G$ of sub-exponential growth rate so that the adjacency operator $H:=A_G$ satisfies
$$
\Bigg(\bigcup_{H'\in\mathcal{R}(H)} \sigma_{\infty}\left(H'\right) \Bigg) \Big\backslash\ \sigma_{\text ess}(H) \neq\emptyset.
$$
\end{prop}
%%%%%%%%%%%%%%%%

\medskip

The proof of Proposition~\ref{prop:strictinclusion} is constructive and the example is sub-exponentially growing but it does not admit uniform sub-exponential growth. If this is assumed then the following holds.

\begin{theorem}
\label{thm:fullresultsubexp}
Let $G$ be an infinite and connected, $d$-bounded graph of uniform sub-exponential growth rate and let $H$ be a bounded Schr\"odinger operator on $\ell^2(G)$. Then,
\begin{align}
\sigma_{\text ess}(H)
	=\bigcup_{H'\in\mathcal{R}(H)}\sigma(H')
	=\bigcup_{H'\in\mathcal{R}(H)}\sigma_\infty(H'). \nonumber
\end{align}

%\end{equation*}
\end{theorem}

As mentioned above, these equalities have been the subject of the previous works \cite{BDE18,E}, which are inspired by \cite{LSess06}. There the first equality is given on graphs of uniform polynomial growth rate \cite[Theorem~2.1]{E}, and on regular trees \cite[Theorem~4]{BDE18}. These results are complemented by \cite[Theorem~3]{BDE18}, by which the first equality is not satisfied in general. The proof of the latter statement includes an example for a graph on which the essential spectrum strictly includes the union over the spectra of the $\mathcal{R}$-limits. While there the growth rate of the graph is not sub-exponential, it can be adapted. Specifically, the construction is similar to the one in the proof of Proposition~\ref{prop:strictinclusion}.

\medskip

As remarked earlier, the equality $\sigma_{\text ess}(H)
	=\bigcup_{H'\in\mathcal{R}(H)}\sigma(H')$ coincides with a recent result from \cite{HaSe19}. The corresponding result there is obtained for metric spaces satisfying a certain set of assumptions. The assumption most relevant for us is known as Property~A \cite{Yu00}, and it is satisfied for graphs of uniform sub-exponential growth rate \cite{Tu01}. Note that the second equality in Theorem~\ref{thm:fullresultsubexp} is mainly part of the current work.

\medskip

In this context it is interesting to mention that, by \cite[Theorem~2.3]{E} Property A is not satisfied for so called uniform graphs of exponential growth rate. While the terminology ``Property A" does not appear there, this property in fact follows, and is related to the earlier argument by \cite{LSess06} for proving Theorem~\ref{thm:fullresultsubexp} for $G=\mathbb{Z}^d$. Moreover, the example mentioned above which is given in the proof of \cite[Theorem~3]{BDE18} is an example of a non-uniform graph on which Property A is not satisfied that can be adjusted to be of sub-exponential growth.

% ------------------- Section II.5 ---------------------------
\section{Examples}
\label{sec:Exam}

The essential spectrum $\sigma_{\text ess}(H)$ is computed here for some examples by using Theorem~\ref{thm:fullresultsubexp}.

\subsection{Variations of \texorpdfstring{$\mathbb{Z}^n$}{Zn}}
The spectrum of an adjacency operator on a graph with bounded vertex degree by $2n$ admitting $\mathbb{Z}^n$ as an $\mathcal{R}$-limit is computed.

\begin{prop}
Let $1<n\in\mathbb{N}$ and let $G$ be a infinite graph of uniform sub-exponential growth such that $A_{\mathbb{Z}^n}$ appears as an $\mathcal{R}$-limit of the adjacency operator $A_G$. If the vertex degree of $G$ is bounded by $2n$ then
\[
\sigma(A_G)=\sigma_{\text ess}(A_G)=[-2n,2n].
\]
\end{prop}

\begin{proof}
A short computation invoking the Cauchy-Schwarz inequality leads to
$$
\Vert A_G f\Vert^2
	= \sum_{v\in G}\left(\sum_{u\sim v}f(u)\right)^2
	\leq 2n \sum_{v\in G}\left(\sum_{u\sim v}(f(u))^2\right)
	\leq(2n)^2 \sum_{u\in G}(f(u))^2,
$$
since in the last estimate each term is positive and appears at most $2n$ times in the total sum. Hence, the spectral radius $\rho(A_G)$ satisfies $\rho(A_G)\leq 2n$ implying $\sigma(A_G)\subseteq [-2n,2n]$.

\medskip

For the converse inclusion, Theorem~\ref{thm:fullresultsubexp} together with \cite[Eq.~(7.3)]{MW89} assert $[-2n,2n]=\sigma\left(A_{\mathbb{Z}^n}\right)\subseteq\sigma_{\text ess}(A_G)$ since $\mathbb{Z}^n$ is an $\mathcal{R}$-limit of $A_G$ and $G$ is of uniform sub-exponential growth.
\end{proof}

For any $n\in\mathbb{N}$ we shall construct a graph which we denote by $Z_{n\times n}$ and is an example for a graph of this family of variations of $\mathbb{Z}_n$.
The construction procedure is the following:

\begin{figure}[ht!]
\centering
\includegraphics[width=100mm]{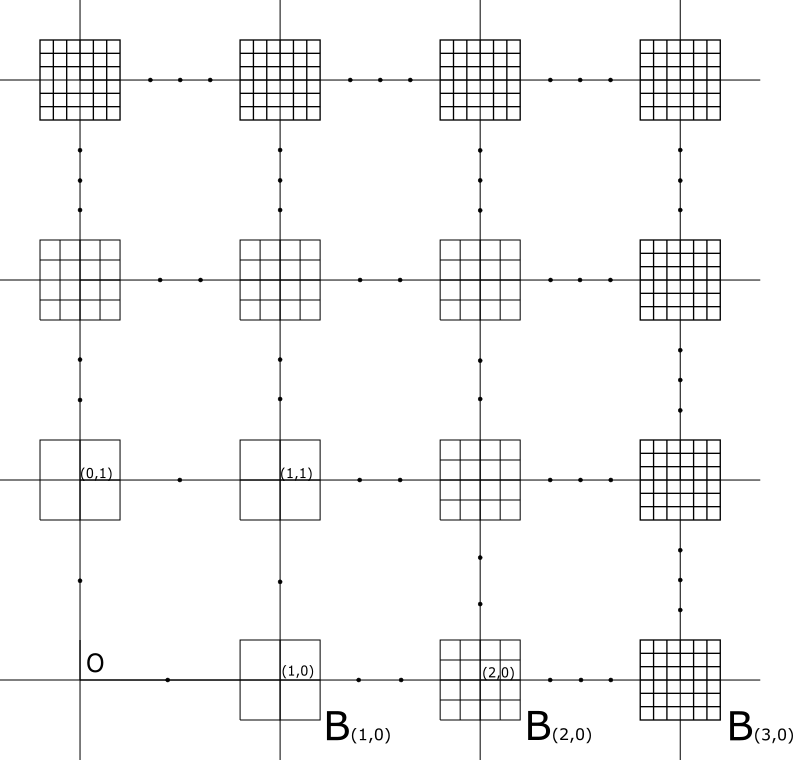}\\
\caption{The graph $Z_{2\times2}$. \label{figure:Z_2times2}}
\end{figure}

\begin{itemize}
\item Denote by $B^n_L$ the subgraph of $\mathbb{Z}^n$ which is the restriction to the box of side length $2L+1$, centred at $0$.
\item For each point $x\in\mathbb{Z}^n$ we shall associate the graph $B_x^n\equiv B_{\Vert x \Vert_\infty}^n$.
\item We connect each adjacent pair of boxes $B_x^n$ and $B_{x+e_j}^n$ by a line. The connection is done between the center points of the corresponding boundary surfaces and includes a sequence of vertices and edges of length $\max\left(\Vert x\Vert_\infty,\Vert x+e_j\Vert_\infty\right)$.
\end{itemize}
For example (a portion of) the graph $Z_{2\times 2}$ is drawn in Figure~\ref{figure:Z_2times2}.
We conclude from the argument above that
\[\sigma\left(A_{Z_{n\times n}}\right)=\sigma_{\text ess}\left(A_{Z_{n\times n}}\right)=[-2n,2n].\]

\subsection{Sparse trees with sparse cycles} \label{sec:example2}
In \cite{Bre07}, so called sparse spherically homogeneous rooted trees were studied. It was shown there that under suitable assumptions the spectrum is purely singular continuous. These graphs are adjusted here by adding from time to time a circle in the graph while preserving the spherical symmetry. Invoking the current result, we compute the spectrum. We only provide a short discussion of these graphs and we refer the reader interested in more details to \cite{E}.

\medskip

%Note that $\mathbb{N}$ denotes all natural numbers $n\geq 1$ and $\mathbb{N}_0$ denotes the set of all natural numbers including $0$. 
A rooted tree $(T,v_0)$ is called {\em spherically homogeneous} if each vertex $v$ is connected with $\kappa(|v|+1)$ vertices with distance $|v|+1$ from the root $v_0$. %If $\kappa(j)<\infty$, the sequence $\{\kappa(j)\}_{j\in\mathbb{N}}$ defines uniquely a tree. 
Let $\{L_n\}_{n\in\mathbb{N}}$ be a strictly increasing sequence with $L_n\in\mathbb{N}$ and $\{k_n\}_{n\in\mathbb{N}}$ be a bounded sequences with $k_n\in\mathbb{N}$ and $k_n>1$. Following \cite{Bre07}, a spherically homogeneous tree is called of type $\{L_n,k_n\}_{n\in\mathbb{N}}$ if $\kappa:\mathbb{N} \to \mathbb{N}$ is defined by
$$
\kappa(j) := \begin{cases}
		k_j,\qquad 	&j\in\{L_n\}_{n\in\mathbb{N}},\\
		1,\qquad	&\text{otherwise}.
	\end{cases}
$$
This graph is called {\em sparse} if $\lim_{n\to\infty} (L_{n+1}-L_n) = \infty$. 

\medskip

Suppose $(T,v_0)$ is a sparse spherically homogeneous tree of type $\{L_n,k_n\}_{n\in\mathbb{N}}$. Let $\{C_n\}_{n\in\mathbb{N}}$ be a sequence of natural numbers satisfying
$$
L_n \geq C_n \geq L_{n-1},\quad
\lim_{n\to\infty} (L_{n}-C_n) = \infty\quad
\text{and}\quad
\lim_{n\to\infty} (C_{n}-L_{n-1}) = \infty.
$$
With this at hand, the {\em sparse tree with sparse cycles} $(G,v_0)$ of type $\{L_n,k_n,C_n\}_{n\in\mathbb{N}}$ is defined based on the spherically homogeneous tree $(T,v_0)$ of type $\{L_n,k_n\}_{n\in\mathbb{N}}$ by adding edges for each $n\in\mathbb{N}$ between vertices in a sphere $\mathcal{S}_{C_n}(v_0)$ of $(T,v_0)$ in the following way: Let $\mathcal{S}_{C_n}(v_0) = \{u_1,\ldots, u_m\}$ be some fixed ordering of the vertices in the sphere of radius $C_n$ around $v_0$. Then we add the edges $(u_j,u_{j+1})$ for $1\leq j\leq m$ and the edge $(u_m,u_1)$. Specifically, each vertex in the sphere of radius $C_n$ in the graph $(G,v_0)$ is adjacent to exactly two other vertices in the sphere and we create a circle, see e.g.\ a sketch of such a graph in Figure~\ref{figure:sparsesym}.

\begin{figure}[ht!]
\centering
\includegraphics[width=110mm]{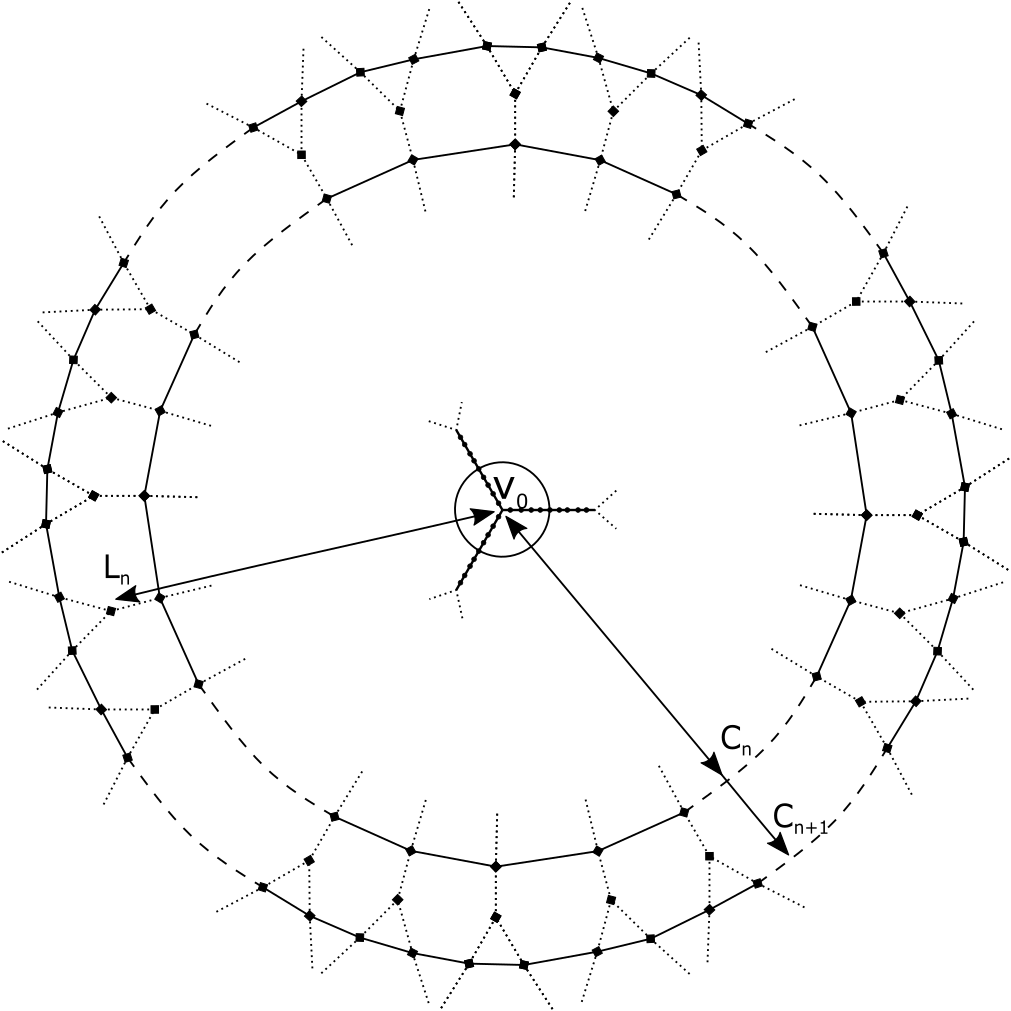}\\
\caption{A sketch of a sparse tree with sparse cycles. \label{figure:sparsesym}}
\end{figure}

In order to apply Theorem~\ref{thm:fullresultsubexp}, the sparse tree with sparse cycles need to be of uniform sub-exponential growth. If $\{L_n\}$ grows exponentially, then a direct computation shows that $(G,v_0)$ is of uniform sub-exponential growth (for instance choose $L_n=10^n$). Denote by $D\subseteq \mathbb{N}$ the set of all accumulation points of the sequence $\{k_n\}_{n\in\mathbb{N}}$, which by construction is finite. Then the possible $\mathcal{R}$-limits of $A$ are the adjacency operators on the following graphs, see Figure~\ref{figure:rlimsparsesym}:
\begin{itemize}
\item The line $\mathbb{Z}$.
\item A two sided infinite comb graph, denoted by $\mathcal{C}$ and defined by
\begin{eqnarray*}
V(\mathcal{C})=&&\left\{v=(k,\ell)\,\Big|\, k,\ell\in\mathbb{Z}\right\}, \\
E(\mathcal{C})=&&\left\{\Big((k,\ell),(k,\ell+1)\Big)\,\Big|\, k,\ell\in\mathbb{Z}\right\}\cup \left\{\Big((k,0),(k+1,0)\Big)\,\Big|\, k\in\mathbb{Z}\right\}.
\end{eqnarray*}
\item The set of star graphs $\{S_{m+1}\}_{m\in D}$, where the star graph denoted by $S_{m+1}$ is defined by m+1 copies of N glued together at $0$. 
\end{itemize}

\begin{figure}[ht!]
\centering
\includegraphics[width=100mm]{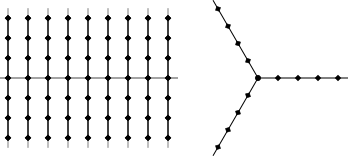}\\
\caption{Some of the $\mathcal{R}$-limits of the sparse tree with sparse cycles (besides $\mathbb{Z}$): the infinite comb graph (left) and the star graph $S_3$ (right). \label{figure:rlimsparsesym}}
\end{figure}

The spectrum of all of these graphs can be explicitly computed. The proofs are following standard ideas and are postponed to the appendix.

\begin{lemma} \label{lem:CGspectra} 
The equality $\sigma(A_{\mathcal{C}})=[-2\sqrt{2},2\sqrt{2}]$ holds.
\end{lemma}

\begin{lemma} \label{lem:SGspectra} 
For $m\in\mathbb{N}$ with $m>1$ we have $\sigma\left(A_{S_m}\right)=[-2,2]\cup\left\{-\frac{m}{\sqrt{m-1}},\frac{m}{\sqrt{m-1}}\right\}$.
\end{lemma}

Thus, Theorem~\ref{thm:fullresultsubexp} implies that the essential spectrum of the adjacency operator associated with the sparse tree with sparse cycles $G$ of uniform sub-exponential growth is given by
\[
\sigma_{\text ess}(A_G)=[-2\sqrt{2},2\sqrt{2}] \cup \bigcup_{m\in D} \left\{-\frac{m+1}{\sqrt{m}},\frac{m+1}{\sqrt{m}}\right\}.
\]

% ------------------- Section III.2 ---------------------------
\section{Characterizing \texorpdfstring{$\sigma_{\text ess}(H)$}{sigma-ess(H)} using generalized eigenfunctions}
\label{sec:CharEssSpec}

The section is devoted to the proof of Theorem~\ref{thm:sigmainfty}. The first part includes a more detailed review on $\mathcal{R}$-limits, and several related properties which we develop and use in this paper. The next three parts include the main tools which we use to show the existence of generalized eigenfunctions. The last part of this section is the actual proof of the theorem.

\subsection{\texorpdfstring{$\mathcal{R}$}{R}-limits}
\label{ssect:R-limits}
If $\eta:A\to B$ is a bijective map on two finite sets $A$ and $B$, denote by {$\mathcal{I}_\eta:\ell^2(A)\cong\mathbb{C}^{\sharp A}\to\ell^2(B)\cong\mathbb{C}^{\sharp B}$} the isomorphism defined via $\mathcal{I}_\eta(\delta_a):=\delta_{\eta(a)}$. Such maps will be mainly used for sets $A$ and $B$ that are balls in different graphs. Specifically, let $(G,v_0)$ be an infinite, connected, rooted $d$-bounded graph. 
Since we assume that the graph $G$ admits a uniform bound on the vertex degree, each ball $B_r(v)$ is finite (for any $v\in V(G)$, $r\in\mathbb{N}$). Throughout this work, $B_r(v)$ defines a subgraph of $G$ by restricting the edge set only to those that connect to vertices in $B_r(v)$. For the sake of simplifying the notation, this induced subgraph is also denoted by $B_r(v)$. 

\medskip

Recall that a bijective map $\phi:V(G)\to V(G')$ between two graphs $G$ and $G'$ (finite or infinite) is called a {\em graph isomorphism} if the induced map $\phi_E:E(G)\to E(G'),\, (u,v)\mapsto \big(\phi(u),\phi(v)\big)$ is also bijective. Then two graphs $G$ and $G'$ are {\em isomorphic ($G\sim G'$)} if there exists a graph isomorphism between them. 
Clearly, $deg(v)=deg(\phi(v))$ holds for all $v\in V(G)$ where $\phi$ denotes the graph isomorphism. Let $G,G'$ be two graphs. If there is an isomorphism between them, then $G$ is a connected $d$-bounded graph if and only if $G'$ is a connected $d$-bounded graph. 
We say that two balls $B_r(u)$ and $B_r(u')$ for $u\in V(G)$ and $u'\in V(G')$ are {\em isomorphic ($B_r(u)\sim B_r(u')$)} if the corresponding subgraph $B_r(u)$ is isomorphic to the subgraph $B_r(u')$. 
%It is immediate to show that if $G$ is a connected $d$-bounded graph and $G'$ is isomorphic to $G$ then $G'$ is also a connected $d$-bounded graph.  Repetition

\medskip

Define the projection
$$
P_{v_0,r}:\ell^2(G)\to \ell^2\big(B_r(v_0)\big)\,,\quad
\big(P_{v_0,r}\psi\big)(v):=\chi_{B_r(v_0)}(v)\psi(v)\,,
$$ 
where $\chi_{B_r(v_0)}$ is the characteristic function of the ball $B_r(v_0)$. Note furthermore that $\ell^2\big(B_r(v_0)\big)$ is naturally embedded into $\ell^2(G)$ by extending a function with zeros. In the following, we will not distinguish between the finite dimensional space $\ell^2\big(B_r(v_0)\big)$ and its embedding in $\ell^2(G)$. For a Schr\"odinger operator $H$ on the rooted graph $(G,v_0)$ denote by $H_{v_0,r}:\ell^2(B_r(v_0))\to\ell^2(B_r(v_0))$ the operator
\[
H_{v_0,r} := P_{v_0,r} H P_{v_0,r}.
\]
It is worth pointing out that $H_{v_0,r}$ can be represented as a matrix acting on $\mathbb{C}^{\sharp B_r(v_0)}$. With this notion at hand, $\psi:V(G)\to\mathbb{C}$ is a generalized eigenfunction of $H$ with eigenvalue $\lambda$ if and only if $P_{v_0,r}H\psi = \lambda P_{v_0,r}\psi$ for all $r\in\mathbb{N}$.

\medskip

Recall that a sequence of vertices $\{v_n\}_{n\in\mathbb{N}}$ {\em goes to infinity} (or {\em converges to infinity}) if it leaves any finite subset of $V(G)$, or equivalently, if $\lim_{n\to\infty} \text{dist}(v_0,v_n)=\infty$. Throughout this paper we will usually assume, without loss of generalty, that the convergence is monotonically. Let $(G,v_0)$ and $(G',v'_0)$ be two rooted graphs. A sequence of maps $f_r:B_r(v_0)\to B_r(v')$ for $r\in\mathbb{N}$ is called {\em coherent (for $v_0$)} if for $s>r$, the restriction of the map $f_s$ to $B_r(v_0)$ equals $f_r$, namely
$$
f_s(u)=f_r(u)\,,\qquad u\in B_r(v_0).
$$

%%%%%%%%%%%%%%%%
\begin{definition} 
\label{def:Hgraphsconv}
Let $H_n$ be a sequence of Schr\"odinger operators on the connected rooted $d$-bounded graphs $(G_n,v_n)$ for $n\in\mathbb{N}$ and let $H'$ be a Schr\"odinger operator on the connected rooted graph $(G',v'_0)$. Then the sequence $\{(H_n,G_n,v_n)\}_{n\in\mathbb{N}}$ is called {\em convergent to $(H',G',v'_0)$} if the following holds
\begin{itemize}
\item[(C1)] There are coherent maps $\left\{f_{n,r}:B_r(v_n)\to B_r(v_0')\right\}_{r\in\mathbb{N}}$ for each $n\in\mathbb{N}$, such that for every $r\in\mathbb{N}$, there exists an $N_r\in\mathbb{N}$ satisfying that $f_{n,r}:B_r(v_n)\to B_r(v_0')$ is a graph isomorphism for all $n\geq N_r$.
\item[(C2)] For each $r\in\mathbb{N}$,
$$
\lim_{n\to\infty} \big\| \mathcal{I}_{f_{n,r}} P_{v_n,r} H_n P_{v_n,r}\mathcal{I}_{f_{n,r}}^{-1} - H'_{v'_0,r}\big\| = 0.
$$
\end{itemize}
\end{definition}
%%%%%%%%%%%%%%%%

Note that we require in the latter definition that the maps $f_{n,r}$ (in $n\in\mathbb{N}$) are eventually graph isomorphism between the balls $B_r(v_n)$ and $B_r(v'_0)$. However, if $n<N_r$ these maps are not necessarily graph isomorphisms. Thus, for $n\geq N_r$, the map $\mathcal{I}_{f_{n,r}}:\ell^2(B_r(v_n))\to\ell^2(B_r(v_0'))$ is an isomorphism and so $\mathcal{I}_{f_{n,r}}^{-1}$ is well-defined. If only (C1) holds, we call the sequence of connected, rooted $d$-bounded graphs $\{(G_n,v_n)\}_{n\in\mathbb{N}}$ {\em convergent} to the connected, rooted $d$-bounded graph $(G',v'_0)$. %\textcolor{red}{Check the following:} 
Additionally note that in Definition~\ref{def:Hgraphsconv} one can replace (C1) with the requirement that there exist maps $f_{n,r}:B_r(v_n)\to B_r(v_0')$ that are eventually bijective. In this case by (C2) the maps $f_{n,r}$ are eventually graph isomorphism.

%Since we have graph isomorphisms, we derive that if $\{G_n,v_n\}_{n\in\mathbb{N}}$ converges to the rooted graph $(G',v'_0)$ then $G'$ is connected and infinite.
%%%%%%%%%%%%%%%%
\begin{definition} \label{def:Rlimit}
Let $H$ be a Schr\"odinger operator on a graph $G$. Then a Schr\"odinger operator $H'$ on a graph $G'$ is called an {\em $\mathcal{R}$-limit of $H$} if there exist a vertex $v'_0\in V(G')$ and a sequence of vertices $v_n\in V(G)$ that montonically converges to infinity such that
$\{(H,G,v_n)\}_{n\in\mathbb{N}}$ converges to $(H',G',v'_0)$.
\end{definition}
%%%%%%%%%%%%%%%%

Whenever it is necessary to specify the coherent maps, we say that $(H',G',v'_0)$ is an $\mathcal{R}$-limit with respect to the coherent maps $\left\{f_{n,r}:B_r(v_n)\to B_r(v_0')\right\}_{r,n\in\mathbb{N}}$. 

\medskip

Notice that the above definition for $\mathcal{R}$-limits is equivalent to the definition given in \cite{BDE18,E}. Nevertheless the definition is presented here slightly different in order to relate it to the more general notion of convergence of a sequence of Schr\"odinger operators, introduced in Definition~\ref{def:Hgraphsconv}, which will be useful for us in this paper. 
%\red{Correct?}
%\textcolor{red}{Need to show that this is coherent with the definition in your thesis respectively \cite{BDE18}.} V4.1

\medskip

We start with some observation that will be helpful. They are inspired by previous considerations on $\mathbb{N}$ \cite{SimSz11}. 
%\red{Correct?}
%The proof of Theorem~\ref{thm:sigmainfty} requires some preliminary propositions for which we don't have to impose any restriction on the growth rate of the graph. The argument is inspired by the proof of \cite[Theorem~7.2.1]{SimSz11}, which concerns operators on $\ell^{2}\left(\mathbb{N}\right)$.

\medskip

Let $G$ be a graph and $\psi:V(G)\to\mathbb{C}$ be a map. Then the {\em support of $\psi$} is defined by $\text{supp}(\psi):=\{u\in V(G)\,:\, \psi(u)\neq 0\}$. Denote by $\mathcal{C}_c(G)$ the set of all $\psi:V(G)\to\mathbb{C}$ such that $\text{supp}(\psi)$ is finite. Clearly, $\mathcal{C}_c(G)\subseteq\ell^2(G)$ holds.

%%%%%%%%%%%%%%%%
\begin{lemma}
\label{lem:RlimProp}
Let $\{(G_n,v_n)\}_{n\in\mathbb{N}}$ be a sequence of  infinite, connected, rooted and $d$-bounded graphs. Assume that $\{(H_n,G_n,v_n)\}_{n\in\mathbb{N}}$ converges to $(H',G',v'_0)$ and satisfies $\sup_{n\in\mathbb{N}}\|H_n\|\leq C$ for some $C>0$. Then $(G',v'_0)$ is an infinite, connected rooted $d$-bounded graph, $H'$ is a Schr\"odinger operator on $(G',v'_0)$ of the form \eqref{eq:H} and $\|H'\| \leq 2C$ holds.
\end{lemma}
%%%%%%%%%%%%%%%%

%%%%%%%%%%%%%%%%
\begin{proof}
Let $\left\{f_{n,r}:B_r(v_n)\to B_r(v_0')\right\}_{r\in\mathbb{N}}$  be the coherent maps (that are eventually isomorphic) such that $\{(H_n,G_n,v_n)\}_{n\in\mathbb{N}}$ converges with respect to these maps to $(H',G',v'_0)$. The claim that $(G',v'_0)$ is an infinite, connected, rooted $d$-bounded graph follows immediately from the definition and the fact that $(G_n,v_n)$ is an infinite, connected, rooted $d$-bounded graph for each $n\in\mathbb{N}$. 

\medskip

Let $r\in\mathbb{N}$. By definition, there exists an $N_r\in\mathbb{N}$ such that $f_{n,r}:B_r(v_n)\to B_r(v'_0)$ is an isomorphism for all $n\geq N_r$ and 
$$
\lim_{n\to\infty, n\geq N_r}\big\| \mathcal{I}_{f_{n,r}} P_{v_n,r} H_n P_{v_n,r}\mathcal{I}_{f_{n,r}}^{-1} - H'_{v'_0,r}\big\| = 0.
$$
Furthermore, $B_r(v'_0)$ and $B_r(v_n)$ are finite sets of the same cardinality and so $\mathcal{I}_{f_{n,r}} P_{v_n,r} H_n P_{v_n,r}\mathcal{I}_{f_{n,r}}^{-1}$ and $H_{v'_0,r}$ can be represented as matrices that converge in the matrix norm to each other. This is equivalent to the convergence of the coefficients. Since by \eqref{eq:H}, $(H_n\psi)(v)$ depends only on the values of $\psi$ on the neighbours of $v$ and $v$ itself, it is straightforward to show that $H'$ is a Schr\"odinger operator of the form \eqref{eq:H}.

\medskip

Let $\psi\in\mathcal{C}_c(G')$ be such that $\|\psi\|\leq 1$. Then there is an $r\in\mathbb{N}$ such that $\text{supp}(\psi)\subseteq B_{r-1}(v'_0)$. According to Definition~\ref{def:Hgraphsconv}, there exists an $n_0\in\mathbb{N}$ such that 
$$
\big\| \mathcal{I}_{f_{n_0,r}} P_{v_{n_0},r} H_{n_0} P_{v_{n_0},r}\mathcal{I}_{f_{n_0,r}}^{-1} - H_{v'_0,r}\big\| < C.
$$
Since $\text{supp}(\psi)\subseteq B_{r-1}(v'_0)$ ($r-1$ is important here) and $H'$ is a Schr\"odinger operator of the form \eqref{eq:H}, we have
$$
H'\psi = H'_{v_{n_0},r} \psi.
$$
Thus, the previous considerations lead to
\begin{align*}
\|H'\psi\| = &\| H'_{v'_0,r}\psi\|\\
	\leq &\big\|H'_{v'_0,r}\psi - \mathcal{I}_{f_{n_0,r}} P_{v_{n_0},r} H_{n_0} P_{v_{n_0},r}\mathcal{I}_{f_{n_0,r}}^{-1} \psi\big\|
	+ \big\| \mathcal{I}_{f_{n_0,r}} H_{v_{n_0},r}\mathcal{I}_{f_{n_0,r}}^{-1} \psi\big\|\\
	< & 2 C
\end{align*}
as $\|\psi\|\leq 1$ and $\mathcal{I}_{f_{n_0,r}}$ is an isomorphism.
\end{proof}
%%%%%%%%%%%%%%%%

\smallskip{}

Next, we show that the operation of considering the $\mathcal{R}$-limits of $H$ is a contraction in the sense that $\mathcal{R}(\mathcal{R}(H))\subseteq \mathcal{R}(H)$.

%%%%%%%%%%%%%%%%
\begin{lemma} 
\label{lem:RRsubsetR}
Let $(G,v_0)$ be a rooted $d$-bounded graph and $H$ be a Schr\"odinger operator on $\ell^2(G)$ as defined in \eqref{eq:H}. If $\{(H'_m,G'_m,u'_m)\}_m$ is a sequence of $\mathcal{R}$-limits of $H$ that converges in the sense of Definition~\ref{def:Hgraphsconv} to $(\widetilde{H},\widetilde{G},\widetilde{v}_0)$ then $(\widetilde{H},\widetilde{G},\widetilde{v}_0)$ is an $\mathcal{R}$-limit of $H$. In particular, $\mathcal{R}(H')\subseteq \mathcal{R}(H)$ holds for all $\mathcal{R}$-limit $H'$ of $H$, and thus also
$\mathcal{R}(\mathcal{R}(H))\subseteq \mathcal{R}(H)$.
\end{lemma}
%%%%%%%%%%%%%%%%

%%%%%%%%%%%%%%%%
\begin{proof}
By assumption, we have the following
\begin{enumerate}
\item The sequence $\left(H'_m,G'_m,u'_m\right)$ converges to $\left(\widetilde{H},\widetilde{G},\widetilde{v}_0\right)$ along the sequence of coherent maps $\left\{ f_{m,r}':B_r(u'_m)\to B_r(\widetilde{v}_0)\right\} _{m,r\in\mathbb{N}}$ that are eventually graph isomorphisms according to Definition~\ref{def:Hgraphsconv}.
\item For $m\in\mathbb{N}$, there is a sequence $\left\{ u^{(m)}_k\right\}_{k\in\mathbb{N}}\subseteq V(G)$ and a sequence of coherent maps $\left\{ f^{(m)}_{k,r}:B_r(u^{(m)}_k)\to B_r(u'_m)\right\} _{k,r\in\mathbb{N}}$ such that $\{(H,G,u^{(m)}_k)\}_{k\in\mathbb{N}}$ converges to $\left(H'_m,G'_m,u'_m\right)$.
\end{enumerate}

Let $\varepsilon>0$ and $R\in\mathbb{N}$. Invoking (1) there is an $M_R\in\mathbb{N}$ such that $f'_{m,R}:B_R(u'_m) \to B_R(\widetilde{v}_0)$ is a graph isomorphism for $m\geq M_R$ and 
$$
\big\| \mathcal{I}_{f'_{m,R}} \left(H'_m\right)_{u'_m,R}\mathcal{I}_{f'_{m,R}}^{-1} - \widetilde{H}_{\widetilde{v}_0,R}\big\| < \frac{\varepsilon}{2},\qquad m\geq M_R.
$$

Let $m\geq M_R$. Invoking (2) there is an $K(R,m)\in\mathbb{N}$ such that $f^{(m)}_{k,R}:B_R(u^{(m)}_k) \to B_R(u'_m)$ is a graph isomorphism for $k\geq K(R,m)$ and 
$$
\big\| \mathcal{I}_{f^{(m)}_{k,R}} H_{u^{(m)}_k,R}\mathcal{I}_{f^{(m)}_{k,R}}^{-1} - \left(H'_m\right)_{u'_m,R}\big\| < \frac{\varepsilon}{2},\qquad k\geq K(R,m).
$$
For $m\geq M_R$, the map
$$
g_{m,R}:=f'_{n,R}\circ f^{(m)}_{K(R,m),R}:B_R(u^{(m)}_{K(R,m)})\to B_R(\widetilde{v}_0)
$$
is a graph isomorphism and $\mathcal{I}_{f'_{m,R}} \mathcal{I}_{f^{(m)}_{K(R,m),R}} = \mathcal{I}_{g_{m,R}}$.
Thus, we derive
\begin{align*}
&\big\| 
	\mathcal{I}_{g_{m,R}} H_{u^{(m)}_{K(R,m)},R} \mathcal{I}_{g_{m,R}}^{-1}
	- \mathcal{I}_{f'_{m,R}} H'_{u'_m,R}\mathcal{I}_{f'_{m,R}}^{-1}
\big\|\\
=\,
&\big\|
		\mathcal{I}_{f^{(m)}_{K(R,m),R}} H_{u^{(m)}_{K(R,m)},R} \mathcal{I}_{f^{(m)}_{K(R,m),R}}^{-1} 
	- H'_{u'_m,R}
	\big\|
	< \frac{\varepsilon}{2}.
\end{align*}
With this at hand, the triangle inequality leads to
\begin{align*}
\big\| 
	\mathcal{I}_{g_{m,R}} H_{u^{(m)}_{K(R,m)},R} \mathcal{I}_{g_{m,R}}^{-1}- \widetilde{H}_{\widetilde{v}_0,R}
\big\|\leq 	 &\big\| 
		\mathcal{I}_{g_{m,R}} H_{u^{(m)}_{K(R,m)},R} \mathcal{I}_{g_{m,R}}^{-1}
		- \mathcal{I}_{f'_{m,R}} H'_{u'_m,R}\mathcal{I}_{f'_{m,R}}^{-1}
	\big\|\\
	 &+ \big\|
		\mathcal{I}_{f'_{m,R}} H'_{u'_m,R}\mathcal{I}_{f'_{m,R}}^{-1} - \widetilde{H}_{\widetilde{v}_0,R}
	\big\|
	< \varepsilon
\end{align*}
for all $m\geq M_R$. Since $\varepsilon>0$ was arbitrary, $\widetilde{H}$ is an $\mathcal{R}$-limit of $H$.
\end{proof}
%%%%%%%%%%%%%%%%

\smallskip{}

The following statement provides a (sequentially) compactness property of the set of triples $(H,G,v)$ where the operators are uniformly bounded in the operator norm.

%%%%%%%%%%%%%%%%
\begin{lemma}
\label{lem:CompConvGraph}
Let $\{(G_n,v_n)\}_{n\in\mathbb{N}}$ be a sequence of connected, infinite, rooted $d$-bounded graphs and $H_n$ be a sequence of Schr\"odinger operators of the form \eqref{eq:H} on $(G_n,v_n)$ such that $\sup_{n\in\mathbb{N}}\|H_n\|\leq C$ for some $C>0$. Then there exists a Schr\"odinger operator $H'$ on a rooted $d$-bounded graph $(G',v'_0)$ and a subsequence $\{H_{n_k},G_{n_k},v_{n_k}\}_{k\in\mathbb{N}}$ that converges to $(H',G',v'_0)$.
\end{lemma}
%%%%%%%%%%%%%%%%

%%%%%%%%%%%%%%%%
\begin{proof}
By assumption $deg(v)\leq d$ holds for all $v\in V(G_n)$ and $n\in\mathbb{N}$. Thus, for each $r\in\mathbb{N}$, the set 
$$
\mathcal{B}_r:=\{ B_r(v)\,:\, v\in V(G_n),\, n\in\mathbb{N}\}/\sim
$$ 
is finite, where $\sim$ is the equivalence relation induced by graph isomorphism. Then a Cantor diagonalization argument gives a convergent subsequence $\{(G_n,v_n)\}_{n\in\mathbb{N}}$ to a rooted $d$-bounded graph $(G',v'_0)$. By a similar argument and by passing to another subsequence one gets the desired result that $\{\left(H_{n_k},G_{n_k},v_{n_k}\right)\}_{k\in\mathbb{N}}$  converges to $(H',G',v'_0)$. Since these arguments are standard, we only provide a sketch of the proof here. 

\medskip

Let $r=1$. Since $\mathcal{B}_1$ is finite, there is a subsequence $\{N(1,k)\}_{k\in\mathbb{N}}\subseteq\mathbb{N}$ with $N(1,k)\to\infty$ such that $B_1(v_{N(1,k_1)})\sim B_1(v_{N(1,k_2)})$ for all $k_1,k_2\in\mathbb{N}$. Now let $r=2$. By the same argument, there is a subsequence $\{N(2,k)\}_{k\in\mathbb{N}}\subseteq \{N(1,k)\}_{k\in\mathbb{N}}$ such that $B_R(v_{N(2,k_1)})\sim B_R(v_{N(2,k_2)})$ for all $k_1,k_2\in\mathbb{N}$ and $1\leq R\leq 2$. By recursion we get for each $r\in\mathbb{N}$ a subsequence 
\begin{equation}
\label{eq:CompConvGraph}
\{N(r,k)\}_{k\in\mathbb{N}}
	\subseteq \{N(r-1,k)\}_{k\in\mathbb{N}} 
	\subseteq \ldots \subseteq \{N(2,k)\}_{k\in\mathbb{N}} 
	\subseteq \{N(1,k)\}_{k\in\mathbb{N}} \subseteq \mathbb{N}
\end{equation}
such that 
$$
B_R(v_{N(r,k_1)})\sim B_R(v_{N(r,k_2)})\,,\qquad k_1,k_2\in\mathbb{N}\,,\, 1\leq R \leq r.
$$
In particular, due to Equation~\eqref{eq:CompConvGraph}, $B_R(v_{N(r,k)})$ is isomorphic to $B_R(v_{N(R,1)})$ for every $k\in\mathbb{N}$ and $1\leq R \leq r$.

\medskip

In order to define the graph $(G',v'_0)$ it suffices to define $B_R(v'_0)$ for all $R\in\mathbb{N}$ modulo graph isomorphism. Define $B_R(v'_0):=B_R(v_{N(R,1)})$. By construction $B_R(v_{N(R,1)})\sim B_R(v_{N(r,k)})$ holds for all $k\in\mathbb{N}$ and $R\leq r$. Thus, the rooted graph $(G',v'_0)$ is well-defined (up to graph isomorphism). By construction, $(G',v'_0)$ is a connected and infinite (rooted) $d$-bounded graph.

\medskip

We claim that the diagonal sequence $\{G_{N(k,k)},v_{N(k,k)}\}_{k\in\mathbb{N}}$ converges to $(G',v'_0)$. This can be seen as follows. Define $f_{k,R}:B_R(v_{N(k,k)}) \to B_R(v'_0)$ for $k\geq R$ to be the graph isomorphism between $B_R(v_{N(k,k)})$ and $B_R(v'_0)= B_R(v_{N(R,1)})$ which exists by construction as {$k\geq R$}. If $k\leq R$, define $f_{k,R}:B_R(v_{N(k,k)}) \to B_R(v'_0)$ by $f_{k,R}(u):= v'_0$ for all $u\in B_R(v_{N(k,k)})$. By construction, these maps are eventually graph isomorphisms (as $k\to\infty$), namely they satisfy the constraints given in Definition~\ref{def:Hgraphsconv}.

\medskip

By the latter considerations we have shown that there is a subsequence that converges to an infinite and connected (rooted) $d$-bounded graph $(G',v'_0)$. In order to simplify the notation, suppose that $\{(G_n,v_n)\}_{n\in\mathbb{N}}$ converges to $(G',v'_0)$. The operators $H_n$ are uniformly bounded in $n\in\mathbb{N}$. Thus, for fixed $r\in\mathbb{N}$, there is a subsequence $\{n_k\}_{k\in\mathbb{N}}$ such that $\mathcal{I}_{f_{n_k,r}} P_{v_{n_k},r} H_{n_k} P_{v_{n_k},r}\mathcal{I}_{f_{n_k,r}}^{-1}$ converges in norm (using that $\ell^2\big(B_r(v_{n_k})\big)$ is a finite dimensional vector space). By a similar argument as for the graph sequence one can construct (with a Cantor diagonalization argument) an operator $H'$ on $(G',v'_0)$ and a subsequence $\{n_k\}_{k\in\mathbb{N}}$ such that for each $r\in\mathbb{N}$
$$
\lim_{k\to\infty, n_k\geq N_r} \big\| \mathcal{I}_{f_{n_k,r}} P_{v_{n_k},r} H_{n_k} P_{v_{n_k},r}\mathcal{I}_{f_{n_k,r}}^{-1} - H'_{v'_0,r}\big\| = 0
$$
where $N_r\in\mathbb{N}$ is chosen such that $f_{n_k,r}$ defines a graph isomorphism for $n_k\geq N_r$. According to Lemma~\ref{lem:RlimProp}, $H'$ is a bounded Schr\"odinger operator of the form \eqref{eq:H}.
\end{proof}
%%%%%%%%%%%%%%%%

\medskip

\subsection{Existence of bounded generalized eigenfunctions for \texorpdfstring{$\mathcal{R}$}{R}-limits}
This section is devoted to provide conditions such that a bounded generalized eigenfunction with eigenvalue $\lambda$ of an $\mathcal{R}$-limit exists. These are key ingredients for the proof of Theorem~\ref{thm:sigmainfty}.

%%%%%%%%%%%%%%%%
\begin{prop}
\label{prop:Rlimseq}
Let $\left(H,G,v_0\right)$ and $\left\{ \left(H_{n},G_n,v_n\right)\right\}_{n\in\mathbb{N}}$ be such that either
\begin{itemize}
\item[(a)] each $\left(H_{n},G_n,v_n\right)$ is an $\mathcal{R}$-limits of $\left(H,G,v_0\right)$, or
\item[(b)] $H_n=H$, $G_n=G$ and $\{v_n\}_{n\in\mathbb{N}}$ monotonically converges to infinity.
\end{itemize}
Let $\varphi^{(n)}:V(G_n)\to\mathbb{C}$, $\lambda_n\in\mathbb{C}$ and $R_n\in\mathbb{N}$ be such that $\lim\limits_{n\to\infty}\lambda_n=\lambda'$, $\lim\limits_{n\to\infty}R_n=\infty$,
$$
P_{v_o,r}H_n\varphi^{(n)} = \lambda_n P_{v_0,r}\varphi^{(n)}\,,\qquad
	\text{for all \ \ } r\leq R_n\,,\; n\in\mathbb{N}\,,
$$ 
and
\begin{equation}\label{eq:Prop3.6cond}
\max_{u\in B_n(v_n)}\left|\varphi^{(n)}(u)\right|
	\leq C\, \left|\varphi^{(n)}\left(v_n\right)\right|\neq0,
\end{equation}
for some constant $C>0$. Then there exists an $\mathcal{R}$-limit $\left(H',G',v_0'\right)\in\mathcal{R}(H)$ that is a limit of $\left\{ \left(H_{n},G_n,v_n\right)\right\}_{n\in\mathbb{N}}$ and a generalized eigenfunction $0\neq\varphi'\in\ell^{\infty}(G')$ of $H'$ with eigenvalue $\lambda'$.
\end{prop}
%%%%%%%%%%%%%%%%

%%%%%%%%%%%%%%%%
\begin{proof}
Due to Lemma~\ref{lem:RlimProp} and Lemma~\ref{lem:CompConvGraph}, there is no loss of generality (by passing to a subsequence) in assuming that the sequence $\{(H_n,G_n,v_n)\}_{n\in\mathbb{N}}$ converges to $(H',G',v'_0)$ with respect to the coherent maps $\left\{f_{n,r}\right\}_{n,r\in\mathbb{N}}$ in the sense of Definition~\ref{def:Hgraphsconv}. In addition there is no loss of generality (again, by passing to a subsequence) to assume that $f_{n,r}:B_r(v_n)\to B_r(v'_0)$ is a graph isomorphism for all $r\leq n$. Define $\psi^{(n)}:V(G')\to\mathbb{C}$ by
\[
\psi^{(n)}(u) := \frac{1}{\varphi^{(n)}(v_n)}\, \left(\mathcal{I}_{f_{n,n}} P_{v_n,n}\varphi^{(n)}\right)(u)
\]
Since $f_{n,n}^{-1}(v'_0)=v_n$, we have $\psi^{(n)}\left(v_{0}'\right)=1$ and
\[
\left\Vert\psi^{(n)}\right\Vert=\sup_{u\in G}\left|\psi^{(n)}\left(u\right)\right|=\sup_{u\in B_n(v_0')}\left|\psi^{(n)}\left(u\right)\right|\leq C
\]
by the assumption \eqref{eq:Prop3.6cond}. %on the sequence of generalized eigenfunctions. 
Thus, $\psi^{(n)}\in\ell^\infty(G')$.

\medskip

By construction, the sequence $\psi^{(n)}(u)$ is uniformly bounded for every $u\in V(G')$. Hence, the Bolzano-Weierstrass theorem and a Cantor diagonalization argument yield that there is a subsequence $\{\psi^{(n(\ell))}\}_{\ell\in\mathbb{N}}\subseteq\mathbb{N}$ and a $\psi':V(G')\to\mathbb{C}$ such that for all $r\in\mathbb{N}$,
\[
\lim_{\ell\to\infty} \left\| P_{v'_0,r}\left(\psi^{(n(\ell))}-\psi'\right)\right\| = 0.
\]
Furthermore, $|\psi'(u)|\leq C$ holds for all $u\in V(G')$, namely $\psi'\in\ell^\infty(G')$. Furthermore, $\psi'(v'_0)=1$ follows from $\psi^{(n)}\left(v_{0}'\right)=1$.

\medskip

In case (a), each $H_n$ is an $\mathcal{R}$-limit of $H$. Lemma~\ref{lem:RlimProp} and Lemma~\ref{lem:RRsubsetR} assert that $H'$ is also an $\mathcal{R}$-limit of $H$ with $\|H'\|\leq 2\|H\|$. If (b) holds, $H'$ is also an $\mathcal{R}$-limit of $H$ as $\{v_n\}_{n\in\mathbb{N}}$ goes to infinity. Thus, it is left to show that $\psi'\in\ell^\infty(G')$ defines a generalized eigenfunction of $H'$ with eigenvalue $\lambda'$. 

\medskip

Let $\varepsilon>0$ and fix $r\in\mathbb{N}$.
%and $\varphi\in\mathcal{C}_c(G')$ be such that $\|\varphi\|=1$, then there exists an $r\in\mathbb{N}$ such that $\text{supp}(\varphi)\subseteq B_{r-1}(v'_0)$. 
Since $H'$ is an $\mathcal{R}$-limit of $H$, it is of the form \eqref{eq:H} (see Lemma~\ref{lem:RlimProp}). Thus, $P_{v'_0,r}H'\psi' = P_{v'_0,r}P_{v'_0,r+1}H'P_{v'_0,r+1}\psi'$ follows implying 
\begin{align*}
\big\|P_{v'_0,r}\left(H'\psi'-\lambda \psi'\right)\big\| 
	\leq \|P_{v'_0,r}\| \big\| H'_{v_0',r+1}\psi' - \lambda' P_{v'_0,r+1}\psi' \big\|
	\leq \big\| H'_{v_0',r+1}\psi' - \lambda' P_{v'_0,r+1}\psi' \big\|.
\end{align*}
In order to simplify the notation, set 
$$
L_{n,r}
	:=\mathcal{I}_{f_{n,r}} P_{v_{n},r} H_{n} P_{v_{n},r}\mathcal{I}_{f_{n,r}}^{-1}
$$
acting on $\ell^2\big(B_{r+1}(v'_0)\big)$. Since $\lim_{\ell\to\infty}\lambda_{n(\ell)}=\lambda'$, there exists a $C_0>0$ such that $|\lambda_{n(\ell)}|\leq C_0$. Recall that for each $n\in\mathbb{N}$, there is an $R_n\in\mathbb{N}$ such that $P_{v_0,r}H_n\varphi^{(n)} = \lambda_n P_{v_0,r}\varphi^{(n)}$ for all $r\leq R_n$.  
Choose $\ell\in\mathbb{N}$ such that $n(\ell), R_{n(\ell)}>r$ and
\begin{align*}
\left\| L_{n(\ell),r+1} -H'_{v_0',r+1}\right\| 
	< &\frac{\varepsilon}{3 \|P_{v'_0,r} \psi'\|}\,,\\
\left\| P_{v'_0,r+1} \big(\psi' - \psi^{(n(\ell))}\big) \right\|
	< &\frac{\varepsilon}{6 \max\{\|H\|,C_0\}}\,,\\
|\lambda_{n(\ell)} - \lambda'|
	< &\frac{\varepsilon}{6 \|P_{v'_0,r} \psi'\|}.
\end{align*}
Note that $\|P_{v'_0,r} \psi'\| \neq 0$ as $\psi'(v'_0)=1$ by construction. 
Furthermore, $\big\| P_{v'_0,r} L_{n(\ell),r+1}\big\|\leq 2\|H\|$ follows as $H_n$ is an $\mathcal{R}$-limit and so $\|H_n\|\leq 2\|H\|$ holds by Lemma~\ref{lem:RlimProp}.
Hence,
\begin{align*}
&\big\|P_{v'_0,r}\left(H'\psi'-\lambda \psi'\right)\big\|\\
	\leq &\big\| P_{v'_0,r} H'_{v_0',r+1}\psi' - \lambda' P_{v'_0,r+1}\psi' \big\|\\
	\leq & \big\|P_{v'_0,r} H'_{v_0',r+1}\psi' - P_{v'_0,r} L_{n(\ell),r+1} \psi'\big\| 
			+ \big\|P_{v'_0,r} L_{n(\ell),r+1} \psi' - P_{v'_0,r} L_{n(\ell),r+1} \psi^{(n(\ell))}\big\|\\
		&+ \big\| P_{v'_0,r} L_{n(\ell),r+1} \psi^{(n(\ell))} - \lambda_{n(\ell)} P_{v'_0,r}\psi' \big\|
			+\big\| \lambda_{n(\ell)} P_{v'_0,r}\psi' - \lambda' P_{v'_0,r+1}\psi' \big\|\\
	=: &(1) + (2) + (3) +(4)
\end{align*}
follows by using the triangle inequality. We estimate each of the summands (1), (2), (3) and (4) separately. Specifically, the previous considerations and the choice of $\ell$ lead to
$$
(1) \leq \left\| L_{n(\ell),r+1} -H'_{v_0',r+1}\right\|  \, \|P_{v'_0,r+1} \psi'\| 
	\leq \frac{\varepsilon}{3}
$$
and
$$
(2) \leq \big\| P_{v'_0,r} L_{n(\ell),r+1}\big\|\, \big\|P_{v'_0,r+1} \big(\psi' - \psi^{(n(\ell))}\big) \big\| 
	\leq \frac{\varepsilon}{3}.
$$
Since $R_{n(\ell)}\geq r$, we have $P_{v_0,r} H_n\varphi^{(n(\ell))}=\lambda_n P_{v_0,r}\varphi^{(n(\ell))}$. Hence, using the choice $n(\ell)>r$, 
$$
P_{v'_0,r} L_{n(\ell),r+1} \psi^{(n(\ell))} = \lambda_{n(\ell)} P_{v'_0,r}\psi^{(n(\ell))} 
$$ %\textcolor{red}{\; check} V4.1 this is good
follows. Thus, 
$$
(3) = \big\| \lambda_{n(\ell)}P_{v'_0,r}\psi^{(n(\ell))} - \lambda_{n(\ell)} P_{v'_0,r}\psi' \big\| < \frac{\varepsilon}{6}.
$$
Finally, using once more the choice of $\ell$, we deduce
$$
(4) \leq |\lambda_{n(\ell)} - \lambda'|\, \|P_{v'_0,r} \psi'\| < \frac{\varepsilon}{6}.
$$
Combining the latter estimates, we conclude
$$
\big\|P_{v'_0,r}\left(H'\psi'-\lambda \psi'\right)\big\| < (1)+(2)+(3)+(4)< \varepsilon
$$
Since $\varepsilon>0$ was arbitrary, we derive that $\big\|P_{v'_0,r}\big(H'\psi'-\lambda'\psi'\big)\big\|=0$ for any $r\in\mathbb{N}$, namely $\lambda'$ is a generalized eigenvalue of $H'$ with generalized eigenfunction $\psi'\in\ell^\infty(G')$.
\end{proof}
%%%%%%%%%%%%%%%%
\smallskip{}

%%%%%%%%%%%%%%%%
\begin{corollary}
\label{cor:UnBoundSpec-Clos}
The union $\bigcup_{H'\in\mathcal{R}(H)}\sigma_\infty(H')$ is closed.
\end{corollary}
%%%%%%%%%%%%%%%%

%%%%%%%%%%%%%%%%
\begin{proof}
Assume $\left\{\lambda_n\right\}_{n\in\mathbb{N}}\subset\bigcup_{H'\in\mathcal{R}(H)}\sigma_\infty(H')$ and $\lim_{n\to\infty}\lambda_n=\lambda'$. For $n\in\mathbb{N}$, let $H_n$ be an $\mathcal{R}$-limit on the $d$-bounded graph $G_n$ and $0\neq\varphi_n\in \ell^\infty(G_n)$ be such that $(H_n-\lambda_n)\varphi_n=0$. Choose $u_n\in V(G_n)$ such that
$\varphi_n(u_n)\geq\frac{\left\Vert \varphi_n\right\Vert_\infty}{2}.$
Then the conditions of Proposition~\ref{prop:Rlimseq} are satisfied (for $C=2$) and so we derive that $\lambda'\in \bigcup_{H'\in\mathcal{R}(H)}\sigma_\infty(H')$.
\end{proof}
%%%%%%%%%%%%%%%%

\medskip

Another consequence of Proposition~\ref{prop:Rlimseq} is the following statement. Let $\{\varphi_i\}_{i\in\mathbb{N}}$ be a sequence of functions $\varphi_i:V(G)\to\mathbb{C}$. We will use the notation
$$
q_{\ell,k}^{(i)}:= \max_{\ell(k-1)\leq |u| <\ell k} |\varphi_i(u)|\,,\qquad i\in\mathbb{N}\,,
$$
and, denote by $u_{\ell,k}^{(i)}\in V(G)$ a vertex satisfying $\ell(k-1)\leq |u_{\ell,k}^{(i)}| <\ell k$ and $q_{\ell,k}^{(i)} = |\varphi_i(u_{\ell,k}^{(i)})|$.

\begin{lemma}
\label{lem-q-Rlimit}
Let $(H,G,v_0)$ be given where $H$ is a Schr\"odinger operator on the infinite connected rooted $d$-bounded graph $(G,v_0)$. Suppose we are given a sequence $\{\varphi_n\}_{n\in\mathbb{N}}$ of generalized eigenfunctions of $H$ with eigenvalues $\{\lambda_n\}_{n\in\mathbb{N}}$ satisfying%such that $0\neq\varphi_j\in\ell^{\infty}(G)$, $H\varphi_j=\lambda_j\varphi_j$, and the following holds.
\begin{itemize}
\item[(a)] $\lim_{n\to\infty}\lambda_n=\lambda$;
\item[(b)] there is an $s>1$ such that for each $k\in\mathbb{N}$, there are subsequences $n_i\to\infty$ and $\ell_i\to\infty$ satisfying
$$
q_{\ell_i,k}^{(n_i)}= |\varphi_{n_i}(u_{\ell_i,k}^{(n_i)})| \geq \frac{1}{s} \max\big\{ q_{\ell_i+1,k}^{(n_i)}, q_{\ell_i-1,k}^{(n_i)} \big\}.
$$
\end{itemize}
Then there is an $\mathcal{R}$-limit $(H',G',v'_0)$ of $H$, and a generalized eigenfunction $0\neq\varphi'\in\ell^{\infty}(G')$ of $H'$ with eigenvalue $\lambda$.
\end{lemma}

\begin{proof}
We will prove in the following that for each $k\in\mathbb{N}$, there is an $\mathcal{R}$-limit $\big(H^{(k)},G^{(k)},v_0^{(k)}\big)$ of $H$ together with a $\varphi^{(k)}:V(G^{(k)})\to\mathbb{C}$ and $\lambda^{(k)}:=\lambda$ satisfying all the assumptions of Proposition~\ref{prop:Rlimseq}. This leads to the desired result.

\medskip

Let $k\in\mathbb{N}$. For the sake of simplifying the notation and since $k$ stays fixed until the last step %for the rest 
of the proof, there is no loss of generality (by passing to a subsequence) in assuming that $n_i=i$ and $\ell_i\geq k$ for all $i\in\mathbb{N}$. Since $\ell_i\to\infty$, the vertices $u_{\ell_i,k}^{(i)}$ go to infinity if $i\to\infty$. Then there is no loss of generality in assuming that $\{(H,G,u_{\ell_i,k}^{(i)})\}_{i\in\mathbb{N}}$ converges to an $\mathcal{R}$-limit $\big(H^{(k)},G^{(k)},v_0^{(k)}\big)$, otherwise we can pass to a convergent subsequence by Lemma~\ref{lem:CompConvGraph}. Since $\ell_i\geq k> k-1$, the requirement (b) yields
$$
|\varphi_i(u)| \leq s |\varphi_i(u_{\ell_i,k}^{(i)})|
$$
for all $u\in V(G)$ satisfying $(\ell_i-1)(k-1) \leq |u| < (\ell_i+1)k$. Let $u\in B_{k-1}(u_{\ell_i,k}^{(i)})$. Since $\ell_i(k-1)\leq |u_{\ell_i,k}^{(i)}| < \ell_ik$ holds by definition, we conclude
$$
(\ell_i-1)(k-1) \leq |u_{\ell_i,k}^{(i)}| - (k-1) \leq |u| \leq |u_{\ell_i,k}^{(i)}| + (k-1) < (\ell_i+1)k.
$$
Hence, 
$$
\max_{u\in B_{k-1}(u_{\ell_i,k}^{(i)})} |\varphi_i(u)| \leq s |\varphi_i(u_{\ell_i,k}^{(i)})|
$$
follows. Define $\psi^{(k)}_i:V(G)\to \mathbb{C}$ by
$$
\psi^{(k)}_i(u) :=	
	\begin{cases}
		\frac{\varphi_i(u)}{|\varphi_i(u_{\ell_i,k}^{(i)})|}\qquad &\,, u\in B_{k-1}(u_{\ell_i,k}^{(i)}),\\
		0\qquad &\,, \text{ otherwise}.
	\end{cases}
$$
Then $\psi^{(k)}_i(u_{\ell_i,k}^{(i)})=1$ and $\|\psi^{(k)}_i\|_\infty \leq s$. Since $\{(H,G,u_{\ell_i,k}^{(i)})\}_{i\in\mathbb{N}}$ converges to $\big(H^{(k)},G^{(k)},v_0^{(k)}\big)$ with respect to the coherent maps $\{f_{i,r}:B_r(u_{\ell_i,k}^{(i)})\to B_r(v_0^{(k)})\}_{i,r\in\mathbb{N}}$, there is an $i_0\in\mathbb{N}$ such that $f_{i,k}$ is a graph isomorphism for all $i\geq i_0$. Define for $i\geq i_0$, $\varphi_i^{(k)}:V(G^{(k)})\to\mathbb{C}$ by $\varphi_i^{(k)}:= \mathcal{I}_{f_{i,k}}\psi_i^{(k)}$. We remind the reader that this formally just defines a function on $\ell^2\big(B_k(v_0^{(k)})\big)$ that we embed into $\ell^2(G^{(k)})$ by extending it by zero. Then we deduce for $i\geq i_0$
$$
\text{supp}(\varphi_i^{(k)})\subseteq B_{k-1}(v_0^{(k)})
	\,,\quad
	\varphi_i^{(k)}\left(v_0^{(k)}\right)=1
	\quad\text{ and }\quad
	\|\varphi_i^{(k)}\|_\infty\leq s.
$$
Since $B_{k-1}(v_0^{(k)})$ is finite, there is no loss of generality (by passing to another subsequence) that $\{\varphi_i^{(k)}\}_{i\in\mathbb{N}}$ converges pointwise (and so in $\ell^2$-norm as all functions are supported on $B_{k-1}(v_0^{(k)})$) to a map $\varphi^{(k)}:V(G^{(k)})\to\mathbb{C}$ such that $\text{supp}(\varphi^{(k)})\subseteq B_{k-1}(v_0^{(k)})$, $\varphi^{(k)}\left(v_0^{(k)}\right)=1$ and $0<\|\varphi^{(k)}\|\leq s$. Thus, 
$$
\sup_{v\in B_k(v_0^{(k)})} |\varphi^{(k)}(v)| \leq s = s \left|\varphi^{(k)}\left(v_0^{(k)}\right)\right|
$$
Now we are almost in the setting of Proposition~\ref{prop:Rlimseq}. More precisely, we need to show that $\varphi^{(k)}$ are ``approximate generalized eigenfunctions''. 

\medskip

Let $r\leq k-2$. Then the convergence of $\{(H,G,u_{\ell_i,k}^{(i)})\}_{i\in\mathbb{N}}$ to $\big(H^{(k)},G^{(k)},v_0^{(k)}\big)$ and the $\ell^2$-convergence of $\{\varphi_i^{(k)}\}_{i\in\mathbb{N}}$ to $\varphi^{(k)}$ imply the following: For all $\varepsilon>0$, there is an $i_1\in\mathbb{N}$ such that for $i\geq i_1$,
\begin{align*}
\left\| \mathcal{I}_{f_{i,r+1}} H_{u_{\ell_i,k}^{(i)},r+1}\mathcal{I}_{f_{i,r+1}} - H^{(k)}_{v_0^{(k)},r+1} \right\| 
	< &\frac{\varepsilon}{2s \|\varphi^{(k)}\|_2},\\
\left\|\varphi_i^{(k)} - \varphi^{(k)}\right\|_2
	< &\frac{\varepsilon}{8 \max\{\|H\|,|\lambda_i|,1\}},\\
\left|\lambda_i - \lambda\right|
	< &\frac{\varepsilon}{8 \|\varphi^{(k)}\|_2}.
\end{align*}
Note that $\varphi^{(k)}$ is supported on $B_{k-1}(v_0^{(k)})$ and so $\|\varphi^{(k)}\|_2$ is finite. Furthermore $\lim_{i\to\infty}\lambda_i=\lambda$ holds implying $|\lambda_i|$ is uniformly bounded in $i\in\mathbb{N}$. According to Lemma~\ref{lem:RlimProp}, $H^{(k)}$ is a Schr\"odinger operator of the form \eqref{eq:H} and $\|H^{(k)}\|\leq 2\|H\|$. Using \eqref{eq:H}, we derive 
$$
P_{v_0^{(k)},r} H^{(k)}\varphi^{(k)} = P_{v_0^{(k)},r} H^{(k)}_{v_0^{(k)},r+1}P_{v_0^{(k)},r+1}\varphi^{(k)}.
$$
Let $v\in B_r(v_0^{(k)})$. Then 
$$
\big(\mathcal{I}_{f_{i,r+1}}P_{u_{\ell_i,k}^{(i)},r+1}\mathcal{I}_{f_{i,k}}^{-1} \zeta\big)(v)
	= \zeta(v),\qquad \zeta\in\ell^2(G^{(k)}),
$$ 
holds as the maps are coherent and $r+1\leq k-1$. Denote by $\chi_{k,r}$ the characteristic function of $B_r(v_0^{(k)})$. Then, a short computation gives
\begin{align*}
\left(P_{v_0^{(k)},r}\mathcal{I}_{f_{i,r+1}}H_{u_{\ell_i,k}^{(i)},r+1}\mathcal{I}_{f_{i,r+1}}^{-1} P_{v_0^{(k)},r+1} \varphi^{(k)}_i\right)(v)
	= & \chi_{k,r}(v) \Big( \mathcal{I}_{f_{i,r+1}}H_{u_{\ell_i,k}^{(i)},r+1}\psi_i^{(k)} \Big)(v)\\
	= & \chi_{k,r}(v)\, \Big(H\psi_i^{(k)}\big)\Big(f^{-1}_{i,r+1}(v)\big)\\
	= & \lambda_i\, \Big(P_{v_0^{(k)},r}\mathcal{I}_{f_{i,r+1}}\psi_i^{(k)}\Big)(v)\\
	= & \lambda_i\, \Big(P_{v_0^{(k)},r}\varphi_i^{(k)}\Big)(v)
\end{align*}
invoking the definition of $\psi_i^{(k)}$ and the fact that $\varphi_i$ is a generalized eigenfunction of $H$ with eigenvalue $\lambda_i$. With this at hand, the triangle inequality leads to
$$
\left\| P_{v_0^{(k)},r} H^{(k)}\varphi^{(k)} -\lambda P_{v_0^{(k)},r}\varphi^{(k)} \right\|
	\leq T_1 + T_2 + T_3 + T_4
$$
where
\begin{align*}
T_1 := &\left\| P_{v_0^{(k)},r} H^{(k)}_{v_0^{(k)},r+1}P_{v_0^{(k)},r+1}\varphi^{(k)} 
		- P_{v_0^{(k)},r}\mathcal{I}_{f_{i,r+1}}H_{u_{l_i,k}^{(i)},r+1}\mathcal{I}_{f_{i,r+1}}^{-1} P_{v_0^{(k)},r+1} \varphi^{(k)}\right\| ,\\
T_2 := & \left\|P_{v_0^{(k)},r}\mathcal{I}_{f_{i,r+1}}H_{u_{l_i,k}^{(i)},r+1}\mathcal{I}_{f_{i,r+1}}^{-1} P_{v_0^{(k)},r+1} \varphi^{(k)} - P_{v_0^{(k)},r}\mathcal{I}_{f_{i,r+1}}H_{u_{l_i,k}^{(i)},r+1}\mathcal{I}_{f_{i,r+1}}^{-1} P_{v_0^{(k)},r+1} \varphi^{(k)}_i\right\|,\\
T_3:= & \left\|P_{v_0^{(k)},r}\mathcal{I}_{f_{i,r+1}}H_{u_{l_i,k}^{(i)},r+1}\mathcal{I}_{f_{i,r+1}}^{-1} P_{v_0^{(k)},r+1} \varphi^{(k)}_i - \lambda_i P_{v_0^{(k)},r}\varphi^{(k)} \right\|,\\
T_4:= & \left\|\lambda_i P_{v_0^{(k)},r}\varphi^{(k)} - \lambda P_{v_0^{(k)},r}\varphi^{(k)} \right\| .
\end{align*}
Then the previous considerations yield
\begin{alignat*}{2}
&T_1< \frac{\varepsilon}{2},\qquad
&T_2\leq 2\|H\| \|\varphi^{(k)} - \varphi^{(k)}_i\| < \frac{\varepsilon}{4} ,\\
&T_3 = |\lambda_i| \big\|P_{v_0^{(k)},r}\varphi_i^{(k)} - P_{v_0^{(k)},r}\varphi^{(k)}\big\| < \frac{\varepsilon}{8},\qquad
&T_4 = |\lambda_i-\lambda| \big\|P_{v_0^{(k)},r}\varphi^{(k)}\big\| <\frac{\varepsilon}{8}
\end{alignat*}
Thus, $\big\| P_{v_0^{(k)},r} H^{(k)}\varphi^{(k)} -\lambda P_{v_0^{(k)},r}\varphi^{(k)} \big\|<\varepsilon$ follows implying
$$
P_{v_0^{(k)},r} H^{(k)}\varphi^{(k)} = \lambda P_{v_0^{(k)},r}\varphi^{(k)}
	\,,\qquad r\leq k-2\,,
$$
as $\varepsilon>0$ was arbitrary.

\medskip

Combining all the previous considerations, Proposition~\ref{prop:Rlimseq} applies for the sequence $\{(H^{(k)},G^{(k)},u^{(k)})\}_{k\in\mathbb{N}}$, the sequence $\{\varphi^{(k)}\}_{k\in\mathbb{N}}$, $R_k:=k-2$, $\lambda_k:=\lambda$ and $C:=s$. Hence, there is an $\mathcal{R}$-limit $(H',G',v'_0)$ of $H$ and a generalized eigenfunction $0\neq \varphi'\in\ell^\infty(G')$ of $H'$ with eigenvalue $\lambda$ finishing the proof.
\end{proof}

\medskip

\subsection{The existence and the behavior at infinity of generalized eigenfunctions for \texorpdfstring{$\mathcal{R}$}{R}-limits}

The next two statements rely on Proposition~\ref{prop:Rlimseq}. We show that under certain conditions the existence of a generalized eigenfunction $\varphi$ of $H$ results in the existence of a bounded generalized eigenfunction of some $\mathcal{R}$-limit of $H$ or it gives constraints on the growth behavior of $\varphi$ at infinity. The first proposition treats bounded generalized eigenfunctions, while in the second proposition they are unbounded. For the proof, an auxiliary lemma is needed for each proposition.
%We start with the proof of Proposition~\ref{prop:bounded}, and we first present a lemma which will be useful.

\begin{lemma} \label{lem:bounded}
Let $s>1$ and $\{a_n\}_{n\in\mathbb{N}}$ be a sequence of non-negative numbers such that $\lim _{n\to\infty} a_n=0$.
Then either there exists $C>0$ such that
\begin{equation} \label{eq:expdec}
a_{n}<C s^{-n}, \qquad n\in\mathbb{N}.
\end{equation}
or there exists a subsequence $\left\{a_{n_k}\right\}_{k=1}^\infty$ such that for any $k\in\mathbb{N}$,
\begin{equation} \label{eq:decsubseq}
a_{n_k}>a_{n_k+1} \quad\text{ and }\quad s a_{n_k}\geq a_{n_k-1}.
\end{equation}
\end{lemma}

\begin{proof}
Define the set
\[
N_1:=\left\{ n\in\mathbb{N}\,|\,a_{n}\geq\max_{k>n}a_{k}\right\}.
\]
Since $\lim_{n\to\infty}a_n=0$, $N_1$ is infinite. Define further,
\[
N_2:=N_1\cap\left\{ n\in\mathbb{N}\,|\,a_{n-1}\leq s a_{n}\right\}.
\]
If $N_2$ is infinite we get a subsequence satisfying \eqref{eq:decsubseq}. Otherwise, $N_2$ is finite. Since $N_1$ is infinite and $N_2$ is finite, there is an $n_0\in N_1$ such that $a_{n-1}> s a_{n}$ for all $n\geq n_0$. Let $n\in N_1$ be such that $n\geq n_0$. Then
$$
a_n< s a_n<a_{n-1}
$$
follows implying $n-1\in N_1$. If $n-1\geq n_0$, one derives similarly $a_{n-2}>s a_{n-1}>r^2 a_n$. Recursively, we get $a_{n_0}>s^{n-n_0}a_n$ implying
$$
a_n < a_{n_0} s^{n_0} s^{-n}\,,\qquad n\geq n_0.
$$
Thus, \eqref{eq:expdec} follows proving the statement.
\end{proof}

\begin{prop} \label{prop:bounded}
Let $H$ be a Schr\"odinger operator on an infinite, connected rooted $d$-bounded graph $(G,v_0)$ and $\varphi\in\ell^\infty(G)$ be a bounded generalized eigenfunction of $H$ with eigenvalue $\lambda$ that does not vanish everywhere. Then one of the following assertions holds.
\begin{enumerate}
\item[(a)] There exists an $\mathcal{R}$-limit $(H',G',v'_0)$ of $H$, and a bounded generalized eigenfunction $0\neq\varphi'\in\ell^{\infty}(G')$ of $H'$ with eigenvalue $\lambda$.
\item[(b)] There are constants $\gamma>1$ and $C>0$ such that 
$$
|\varphi\left(u\right)|\leq C\cdot \gamma^{-\left|u\right|}
	\,,\qquad u\in G\,.
$$
\end{enumerate}
\end{prop}

\begin{proof}
Let $k\in\mathbb{N}$. Then there is a $u_{m,k}\in V(G)$ such that $m(k-1)\leq |u_{m,k}| < mk$ and
\[
|\varphi(u_{m,k})| = \max_{m\left(k-1\right)\leq\left|u\right|<mk}\left|\varphi\left(u\right)\right|
	=: q_{m,k}
\]
Note that the latter is a maximum as $B_r(v_0)$ is finite for every $r\in\mathbb{N}$. We will treat the two cases where \textbf{(A)} $\{q_{m,k}\}_{m\in\mathbb{N}}$ does not tend to zero for some $k\in\mathbb{N}$ and \textbf{(B)} $\{q_{m,k}\}_{m\in\mathbb{N}}$ does tend to zero for all $k\in\mathbb{N}$.

\medskip

\textbf{(A)} Suppose $k\in\mathbb{N}$ is chosen such that $\{q_{m,k}\}_{m\in\mathbb{N}}$ does not tend to zero. Let $\left\{q_{m_\ell,k}\right\}_{\ell\in\mathbb{N}}$ be a subsequence such that $\lim_{\ell\to\infty}q_{m_\ell,k}= q>0$. Such a subsequence exists as $\{q_{m,k}\}_{m\in\mathbb{N}}$ is uniformly bounded since $\varphi\in\ell^\infty(G)$.
Choose $\ell_0\in\mathbb{N}$ such that $q_{m_{\ell},k}>q/2$ for every $\ell>\ell_0$. Then,
\[
\max_{u\in B_{\ell}(u_{m_{\ell},k})}
 \left|\varphi(u)\right|
	\leq \left\| \varphi \right\|_\infty 
	\leq \frac{2 \|\varphi\|_\infty}{q}\, \left|\varphi\left(u_{m_{\ell},k}\right)\right|
	\neq0.
\]
Thus, all the requirements of Proposition~\ref{prop:Rlimseq} are satisfied for the sequence $\{(H,G,u_{m_\ell,k})\}_{\ell\in\mathbb{N}}$ with $\lambda_n=\lambda$, $\varphi^{(n)}=\varphi$ and $C:= \frac{2\|\varphi\|_\infty}{q}$. Hence, the statement (a) of Proposition~\ref{prop:bounded} follows from Proposition~\ref{prop:Rlimseq}.

\medskip

\textbf{(B)} Suppose now that $\lim_{m\to\infty}q_{m,k}=0$ for all $k\in\mathbb{N}$ and let $s>1$. Due to Lemma~\ref{lem:bounded}, either \textbf{(B.1)} there is a $k\in\mathbb{N}$ such that \eqref{eq:expdec} holds for a suitable constant $C>0$ and $a_m:=q_{m,k}$ or \textbf{(B.2)} for all $k\in\mathbb{N}$, there is a subsequence $\{a_{m_i}\}_{i\in\mathbb{N}}$ satisfying \eqref{eq:decsubseq}.

\medskip

\textbf{(B.1)} There exists a $k\in\mathbb{N}$ and a constant $C>0$ such that 
\[
q_{m,k}<C s^{-m}\,,\qquad m\in\mathbb{N}.
\]
If $u\in V(G)$, then there is an $m\in\mathbb{N}$ such that $m(k-1)\leq |u| < mk$ and so $-m<-\frac{|u|}{k}$. Consequently, the latter considerations yield
\[
q_{m,k}<C s^{-\frac{|u|}{k}}.
\]
Specifically, $\varphi$ is exponentially decaying as claimed in (b) with $\gamma:=s^{\frac{1}{k}}>1$.

\medskip

\textbf{(B.2)} For all $k\in\mathbb{N}$, there is a subsequence $\left\{q_{m_i,k}\right\}_{i\in\mathbb{N}}$ satisfying
\begin{equation}\label{eq:(prop:bounded)}
q_{m_i,k}>q_{m_i+1,k} 
	\quad\text{ and }\quad 
	s\, q_{m_i,k}\geq q_{m_i-1,k}.
\end{equation}
Hence,
$$
q_{m_i,k} \geq \frac{1}{s} \max\big\{ q_{m_i-1,k} ,\, q_{m_i+1,k} \big\}
$$
follows and $\varphi$ is a generalized eigenfunction of $H$ with eigenvalue $\lambda$. Thus, Lemma~\ref{lem-q-Rlimit} applied to $\varphi_j:=\varphi$ and $\lambda_j:=\lambda$ implies (a).
\end{proof}

Next, unbounded generalized eigenfunctions are studied in Proposition~\ref{prop:unbounded}. For this the following lemma will be useful.

\begin{lemma} \label{lem:unbounded}
Let $\{a_n\}_{n\in\mathbb{N}}$ be an unbounded sequence of non-negative numbers such that there is an $s>1$ and a $C>0$ satisfying
$$
a_n<C s^n\,,\qquad n\in\mathbb{N}.
$$
Then, for each $\gamma>s$, there exists a subsequence $\left\{a_{n_k}\right\}_{k=1}^\infty$ such that for all $k\in\mathbb{N}$,
\[
a_{n_k}>a_{n_k-1} \quad\text{ and }\quad \gamma a_{n_k}\geq a_{n_k+1}.
\]
\end{lemma}

\begin{proof}
Define
\[
N_1=\left\{ n\in\mathbb{N}\,|\,a_{n}\geq\max_{m\leq n}a_{m}\right\} .
\]
Since $a_{n}$ is unbounded,  $\sharp N_1=\infty$.
Define further,
\[
N_2=N_1\cap\left\{ n\in\mathbb{N}\,|\,a_{n+1}\leq \gamma a_{n}\right\}.
\]
Clearly, the statement is proven if $N_2$ is infinite. Assume by contradiction that $N_2$ is finite. Then there exists an $n_0\in\mathbb{N}$ so that for every $n\geq n_0$, $a_{n+1}> \gamma a_{n}$. Since $N_1$ is infinite, there is an $n_1\in N_1$ such that $n_1\geq n_0$. Then a short induction yields that for $n\geq n_1$, we have 
$$
n\in N_1
 \quad \text{ and }\quad
 a_n>\gamma^{n-n_1}a_{n_1}.
$$ 
For indeed, since $n_1\geq n_0$, we get
$$
a_{n_1+1}>\gamma a_{n_1} > a_{n_1} \geq \max_{m\leq n_1}a_{m}.
$$
Thus, $n_1+1\in N_1$ and so the base case is proven. Analogously, the induction step follows.

\medskip

Hence, there is a constant $C_1>0$ such that $a_n> C_1 \gamma^n$ for all $n\in\mathbb{N}$. Since $\gamma>s$, this contradicts the assumption $a_n<C s^n$ for some constant $C>0$. Hence, $N_2$ is infinite.
\end{proof}

%\textcolor{red}{assertion (b) will guarantee that we are not in the spectrum}

\begin{prop} 
\label{prop:unbounded}
Let $\varphi$ be a unbounded generalized eigenfunction of $H$ with eigenvalue $\lambda$. Then one of the following assertions holds.
\begin{enumerate}
\item[(a)] There exists an $\mathcal{R}$-limit $(H',G',v'_0)$ of $H$, and a bounded generalized
eigenfunction $\varphi'\in\ell^{\infty}(G')$ of $H'$ with eigenvalue $\lambda$ that does not vanish everywhere.
\item[(b)] There exists a constant $\gamma>1$, such that for all $C>0$, there is a $u=u(C)\in V(G)$ satisfying $|\varphi\left(u\right)|\geq C\, \gamma^{\left|u\right|}$.
\end{enumerate}
\end{prop}

\begin{proof}
Suppose (b) is not satisfied. We will prove that then (a) holds. Let $k\in\mathbb{N}$. Define 
$$
q_{m,k}
	:=\max_{m\left(k-1\right)\leq\left|u\right|<mk}\left|\varphi\left(u\right)\right|\,,
	\qquad m\in\mathbb{N}.
$$
Since (b) does not hold, for all $\gamma>1$, there is a constant $C_{\gamma,k}>0$ such that
$$
|\varphi(u)|\leq C_{\gamma,k} \gamma^{|u|}\,,\qquad u\in V(G).
$$
Let $s>1$ and set $\gamma:=s^{\frac{1}{k}}>1$. Thus, there is a constant $C_k=C(k,r)>0$ such that
\[
q_{m,k}
	< C_{k}\, \gamma^{mk}
	= C_{k}\, s^m.
\]
Since $\varphi$ is unbounded, $\lim_{m\to\infty}q_{m,k}=\infty$. Then Lemma~\ref{lem:unbounded} implies that there is a subsequence $\left\{q_{m_i,k}\right\}_{i\in\mathbb{N}}$ such that for every $\gamma'>s$,
\[
q_{m_i,k}>q_{m_i-1,k} \quad\text{ and }\quad 
	\gamma' q_{m_i,k}\geq q_{m_i+1,k}.
\]
Thus, Lemma~\ref{lem-q-Rlimit} applied to $\varphi_j:=\varphi$, $\lambda_j:=\lambda$ and $s=\gamma'$ leads to (a).
\end{proof}

\subsection{Generalized eigenfunctions and Shnol type theorems}
\label{ssect:Shnol}
As discussed in the introduction, Shnol type theorems connect a growth conditions for generalized eigenfunctions with eigenvalue $\lambda$ to the fact that $\lambda$ belongs to the spectrum of the operator. 

\medskip

We say that a function $\varphi:V(G)\to\mathbb{C}$ on a graph $G$ has {\em sub-exponential growth (with respect to the graph metric $dist$)} if for one (any) vertex $v_0\in V(G)$, the map
$$
V(G)\ni v\mapsto e^{-\alpha\ \text{dist}(v,v_0)}\ \varphi(v)
$$ 
is an element of $\ell^2(G)$ for all $\alpha>0$. 
With this at hand, the following Shnol-type theorem holds, proven in \cite{HK2011} in more general setting of bounded Jacobi operators on a graph. This theorem is used for the proof of Theorem~\ref{thm:fullresultsubexp}.

\begin{prop}[{\cite[Theorem~4.8]{HK2011}}]\label{prop:HK4.8}
Let $H$ be a Schr\"odinger operator of the form \eqref{eq:H} on an infinite, connected, rooted $d$-bounded graph $(G,v_0)$. Suppose $\varphi:V(G)\to\mathbb{C}$ is a generalized eigenfunction of $H$ with eigenvalue $\lambda$. 
If $\varphi$ is sub-exponentially bounded, then $\lambda \in \sigma(H)$.
\end{prop}

The proof of Theorem~\ref{thm:sigmainfty} depends on the existence of a generalized eigenfunction for each point in the spectrum admitting a suitable growth rate. The following statement provides a sufficient condition to get such generalized eigenfunction which can be found in \cite[Theorem~3]{LT16} in a more general setting.

%The proof relies on the existence of such generalized eigenfunctions of specific growth rate, a property known as a ``reverse" Shnol's Theorem. We shall use the following form of a reverse Shnol's Theorem on graphs, which follows from \cite{LT16} (in which this property is given in a slightly broader context):

\begin{theorem}[reverse Shnol's Theorem, {\cite[Theorem 3]{LT16}}]
\label{thm:revSch} 
Let $(G,v_0)$ be a connected, infinite, rooted graph and $H$ be a Schr\"odinger operator on $\ell^{2}(G)$ of the form \eqref{eq:H} with spectral measure $\mu$. Suppose $\omega\in\ell^2(G)$ is real-valued and positive $($i.e.\ $\omega(v)>0$ for all $v\in G)$. Then for $\mu$-a.e.\ $\lambda\in\sigma\left(H\right)$, there exists a generalized eigenfunction $\varphi$ of $H$ with eigenvalue $\lambda$ satisfying $\varphi\cdot\omega\in\ell^2(G)$.
\end{theorem}

\begin{corollary}
\label{cor:revSch}
Let $(G,v_0)$ be a connected, infinite, rooted graph and $H$ be a Schr\"odinger operator on $\ell^{2}(G)$ of the form \eqref{eq:H}. If $\omega\in\ell^2(G)$ is real-valued and positive, then
\[
\overline{\left\{\lambda \in \sigma(H)\,\big|\,\exists\varphi:G\to\mathbb{C},\,H\varphi=\lambda\varphi \text{ and } \varphi\omega\in\ell^2(G)\right\}}=\sigma(H).
\]
\end{corollary}

\begin{lemma} \label{lem-l2}
Let $(G,v_0)$ be an infinite, connected, rooted graph. Then $\omega_G:V(G)\to\mathbb{R}$ defined by
\[
\omega_G(v)
	:=\frac{1}{(|v|+1)\, \sqrt{\sharp \mathcal{S}_{|v|}(v_0)}}
\]
satisfies $\omega_G\in\ell^2(G)$.
\end{lemma}
%%%%%%%%%%%%%%%%

%%%%%%%%%%%%%%%%
\begin{proof}
A short computation gives
%\textcolor{red}{add the proof} V4.1
\[
\Vert\omega_G\Vert_2^2=\sum_{v\in G}|\omega_G(v)|^2=\sum_{k=0}^\infty \sharp \mathcal{S}_k(v_0)\cdot\frac{1}{(k+1)^2 \sharp \mathcal{S}_k(v_0)}=\sum_{k=0}^\infty \frac{1}{(k+1)^2}<\infty.
\]
\end{proof}
%%%%%%%%%%%%%%%%

\smallskip{}

As a consequence of Lemma~\ref{lem-l2} and Theorem~\ref{thm:revSch} we derive

\begin{corollary}
\label{cor:decayGenEigenfct}
The set 
$$
\Lambda_1=
\left\{
	\lambda\in\sigma(H) \,\left|\, 
	\begin{array}{c}
		\text{there is a generalized eigenfunction } \varphi_\lambda 
		\text{ with eigenvalue } \lambda\\[0.1cm]
		\text{ such that } 
		\left|\varphi_\lambda\left(v\right)\right| 
			\leq \left(|v|+1\right)\, \sqrt{\sharp \mathcal{S}_{|v|}(v_0)}		
	\end{array}
	\right.
\right\}
$$
is dense in $\sigma(H)$. In particular, if $G$ has sub-exponential growth, then the set
\begin{equation}\label{eq:Lambda2}
\Lambda_2=
\left\{
	\lambda\in\sigma(H) \,\left|\, 
	\begin{array}{c}
		\text{there is a generalized eigenfunction } \varphi_\lambda 
		\text{ with eigenvalue } \lambda\\[0.1cm]
		\text{ such that }
		\varphi_\lambda \text{ is sub-exponentially growing}
	\end{array}
	\right.
\right\}
\end{equation}
is dense in $\sigma(H)$.
\end{corollary}

\begin{proof}
This follows immediately from Lemma~\ref{lem-l2} and Corollary~\ref{cor:revSch}. The second part follows from the fact that if $G$ is sub-exponentially growing, then for each $\gamma>1$, there is a $C>0$ such that
$$
\left(|v|+1\right)\, \sqrt{\sharp \mathcal{S}_{|v|}(v_0)}	
	\leq C\cdot\gamma^{\left|v\right|/2}\left(|v|+1\right).
$$
Then the statement follows from the first part.
\end{proof}

%For example on a $d$-regular tree we get a generalized eigenfunction satisfying
%for any $v\in T$,
%\begin{equation*} \label{eq:treeGrowthEF}
%\left|\varphi\left(v\right)\right|\leq C\cdot\left(d-1\right)^{\left|v\right|/2}\left(|v|+1\right),
%\end{equation*}
%where $C>0$ is a constant.

\subsection{Proof of Theorem~\ref{thm:sigmainfty}}

The strategy of the proof of Theorem~\ref{thm:sigmainfty} is as follows. Given $\lambda\in\sigma_{\text ess}(H)$ we use the reverse Shnol's property (Theorem~\ref{thm:revSch}) to obtain a sequence of generalized eigenfunctions $\left\{\varphi^{(n)}\right\}_{n=1}^\infty$ of $H$ with eigenvalue $\lambda_n$ such that $\lim_{n\to\infty}\lambda_n=\lambda$ and each $\varphi_n$ is sub-exponentially growing. With this at hand, Proposition~\ref{prop:Rlimseq}, Propositions~\ref{prop:bounded} and Proposition~\ref{prop:unbounded} complete the proof of the theorem.

\medskip

Recall that the discrete spectrum $\sigma_{disc}(H)$ of an operator $H$ is defined by the set of isolated eigenvalues of finite multiplicity. Furthermore, the essential spectrum $\sigma_{ess}(H)$ is defined by $\sigma(H)\setminus\sigma_{disc}(H)$. Thus, if $\lambda$ is an element of the essential spectrum, then either $\lambda$ is an eigenvalue of infinite multiplicity or in each neighbourhood of $\lambda$ there are elements of the spectrum $\sigma(H)$ that are not equal to $\lambda$.

\medskip

\begin{proof}[Proof of Theorem~\ref{thm:sigmainfty}]
Let $\lambda\in\sigma_{\text ess}(H)$. Then one of the following cases holds
\begin{itemize}
\item[(a)] For every $\varepsilon>0$ there is a $\lambda_\varepsilon\in\sigma(H)$ such that $0<|\lambda-\lambda_\varepsilon|<\varepsilon$.
\item[(b)] There exists an infinite sequence $\varphi_n\in\ell^2(G)$ satisfying $H\varphi_n=\lambda\varphi_n$ and they are pairwise orthogonal.
\end{itemize}

Begin with case (a), due to Corollary~\ref{cor:decayGenEigenfct}  and since $(G,v_0)$ has sub-exponential growth, the set $\Lambda_2$ (of \eqref{eq:Lambda2}) 
is dense in $\sigma(H)$. 
Then by the assumption (a) there is a sequence of eigenvalues $\{\lambda_n\}_{n\in\mathbb{N}}$ corresponding to generalized eigenfunctions $\varphi_n:V(G)\to\mathbb{C}$ such that $\lambda_n\neq\lambda_m$ for $m\neq n$, $\lim_{n\to\infty}\lambda_n=\lambda$ and $\varphi_n$ has sub-exponential growth. Then either infinitely many of the $\varphi_n$ are unbounded or only finitely many of them are unbounded

\medskip

If infinitely many of the $\varphi_n$ are unbounded, there is no loss of generality in assuming that each $\varphi_n$ is unbounded (otherwise pass to a subsequence). 
Then Proposition~\ref{prop:unbounded} asserts that $\lambda_n\in\sigma_\infty(H_n)$ for some $\mathcal{R}$-limit $H_n$ of $H$ or there is a $\gamma>1$ such that for all $C>0$, there exists a $u\in V(G)$ with $|\varphi(u)|\geq C\gamma^{|u|}$.
However, the second assertion cannot hold as $\varphi_n$ is sub-exponentially bounded. Hence, $\lambda_n\in\sigma_\infty(H_n)$ holds for some $H_n\in\mathcal{R}(H)$. Since $\lambda_n\to\lambda$, Corollary~\ref{cor:UnBoundSpec-Clos} yields $\lambda\in\bigcup_{H'\in\mathcal{R}(H)}\sigma_\infty(H')$ finishing the proof.

\medskip

If only finitely many of the $\varphi_n$'s are unbounded, there is no loss of generality in assuming that $\varphi_n$ is bounded for each $n\in\mathbb{N}$ (otherwise pass to a subsequence). By Proposition~\ref{prop:bounded} either infinitely many of the $\lambda_n$ satisfy $\lambda_n\in\sigma_\infty(H_n)$ for some $H_n\in\mathcal{R}(H)$ or only finitely many. In the first case Corollary~\ref{cor:UnBoundSpec-Clos} finishes the proof as before. In the second case, Proposition~\ref{prop:bounded} asserts that for $n$ large enough
$$
|\varphi_n\left(u\right)|\leq C_n\cdot \gamma_n^{-\left|u\right|}
	\,,\qquad u\in G\,.
$$
Thus, $\varphi_n\in\ell^2(G)$ and in particular, they are orthogonal as $H$ is self-adjoint and $\lambda_n\neq\lambda_m$ for $m\neq n$. This is the same situation as in (b) (where $\lambda_n\equiv\lambda$ for all $n\in\mathbb{N}$). We proceed proving the claim of the theorem in this last case which will also prove it for the case (b).

\medskip

Set $\psi_n:= \frac{\varphi_n}{\|\varphi_n\|_2}$ for each $n\in\mathbb{N}$. Then $\psi_n$ is still an eigenfunction of $H$ with eigenvalue $\lambda_n$ ($\lambda_n=\lambda$ in (b)) and the functions $\{\psi_n\}_{n\in\mathbb{N}}$ are pairwise orthogonal. Let $u_n\in V(G)$ be the vertex such that $\|\psi_n\|_\infty=|\psi_n(u_n)|$.  We have two cases, either $\{u_n\}_{n\in\mathbb{N}}$ (monotonically, by passing to a subsequence) converges to infinity or $\{u_n\}_{n\in\mathbb{N}}\subseteq B_r(v)$ for some $r>0$, $v\in V(G)$ and every $n\in\mathbb{N}$. In the first case, when $\{u_n\}_{n\in\mathbb{N}}$ converges to infinity, then
$$
\max_{u\in B_n(u_n)} \big| \psi_n(u) \big| \leq \|\psi_n\|_\infty = |\psi_n(u_n)|
$$
holds by construction. Thus, Proposition~\ref{prop:Rlimseq} yields that there exists an $\mathcal{R}$-limit $(H',G',v'_0)$ being a limit point of $(H,G,u_n)$ and a generalized eigenfunction $0\neq \psi\in\ell^\infty(G')$ of $H'$ with eigenvalue $\lambda$. 

\medskip

It is left to treat the case that $\{u_n\}_{n\in\mathbb{N}}\subseteq B_r(v)$ for some $r>0$, $v\in V(G)$ and every $n\in\mathbb{N}$. %We will show that this always leads to a contradiction, finishing the proof. 
Since $\{\psi_n\}_{n\in\mathbb{N}}$ are pairwise orthogonal, they converge weakly to zero, see e.g.\ \cite[Theorem II.6]{ReedSimonI}. Since $\|\psi_n\|_\infty = |\psi_n(u_n)|$ and $u_n\in B_r(v)$ for all $n\in\mathbb{N}$, the weak convergence to zero yields $\lim_{n\to\infty}\|\psi_n\|_\infty=0$. For $k\in\mathbb{N}$ define
$$
q_{\ell,k}^{(n)} := \max_{\ell(k-1)\leq |u| < \ell k} |\psi_n(u)|.
$$
Since $\psi_n\in\ell^2(G)$, we derive $\lim_{\ell\to\infty} q_{\ell,k}^{(n)}=0$. Define 
$$
L_{n,k}:=\big\{\ell\in\mathbb{N}\,:\, q_{\ell,k}^{(n)} \geq \max_{i\geq \ell} q_{i,k}^{(n)} \big\}.
$$
Since $\lim_{\ell\to\infty} q_{\ell,k}^{(n)}=0$, we conclude that $L_{n,k}$ is infinite for all $n\in\mathbb{N}$. Next, we will prove the following lemma that together with Lemma~\ref{lem-q-Rlimit} will conclude the proof.

\medskip

\begin{lemma}\label{lem:lastlem}
Suppose we are in the setting as described before. For each $k\in\mathbb{N}$ there are two sequences $\{n_j\}_{j\in\mathbb{N}}$ (with $n_j\to\infty$) and $\{\ell_j\}_{j\in\mathbb{N}}$ (with $\ell_j\to\infty$) such that for each $j\in\mathbb{N}$, $\ell_j\in L_{n_j,k}$ and $q_{\ell_j-1,k}^{(n_j)} \leq 2q_{\ell_j,k}^{(n_j)}$.
\end{lemma}

\medskip

Before proving Lemma~\ref{lem:lastlem}, let us explain why it finishes the proof. By construction $\psi_n$ is a generalized eigenfunction of $H$ with eigenvalue $\lambda_n$ satisfying $0\neq \psi_n\in\ell^\infty(G)$ and $\lim_{n\to\infty}\lambda_n=\lambda$. Lemma~\ref{lem:lastlem} implies that for each $k\in\mathbb{N}$, there are subsequences $n_j\to\infty$ and $\ell_j\to\infty$ satisfying $\ell_j\in L_{n_j,k}$ and $q_{\ell_j-1,k}^{(n_j)} \leq 2q_{\ell_j,k}^{(n_j)}$. The condition $\ell_j\in L_{n_j,k}$ leads to $q_{\ell_j,k}^{(n_j)}\geq q_{\ell_j+1,k}^{(n_j)}$. Thus, the constraint (b) of Lemma~\ref{lem-q-Rlimit} is satisfied with $s=2$. Hence, this lemma implies that there is an $\mathcal{R}$-limit $(H',G',v'_0)$ of $H$ and a generalized eigenfunction $0\neq \varphi'\in\ell^\infty(G')$ of $H'$ with eigenvalue $\lambda$ proving Theorem~\ref{thm:sigmainfty}.
\end{proof}

\medskip

\begin{proof} [Proof of Lemma~\ref{lem:lastlem}]
Let us fix $k \in\mathbb{N}$ and introduce the following notation $q_\ell^{(n)} := q_{\ell,k}^{(n)}$ and $L_n:=L_{n,k}$. Assume by contradiction that 
\begin{equation}
\label{eq:(1)}
\forall \{n_j\}_j,\; \forall \{\ell_j\}_j \text{(with } n_j\to\infty, \ell_j\to\infty), \exists j\in\mathbb{N} \,: 
	\;\; \ell_j\not\in L_{n_j} \;\; \text{ or }\;\; q_{\ell_j-1}^{(n_j)} > 2q_{l_j}^{(n_j)}.
\end{equation}
We first prove that this is equivalent to
\begin{equation}
\label{eq:(2)}
\exists n_0,\ell_0\in\mathbb{N},\quad \forall n\geq n_0,\; \forall \ell\geq \ell_0\,:\quad  q_\ell^{(n)} < \max_{i\geq \ell} q_i^{(n)} \quad \text{ or } \quad q_{\ell-1}^{(n)} > 2q_{\ell}^{(n)}.
\end{equation}
That \eqref{eq:(2)} implies \eqref{eq:(1)} is straightforward. In order to show the implication \eqref{eq:(1)}$\Rightarrow$\eqref{eq:(2)}, assume by contradiction that \eqref{eq:(1)} holds but \eqref{eq:(2)} does not hold. First note that $\ell\in L_{n}$ is equivalent to  $q_\ell^{(n)} \geq \max_{i\geq \ell} q_i^{(n)}$. Lets denote by $A_{n,\ell}$ the statement $q_\ell^{(n)} \geq \max_{i\geq \ell} q_i^{(n)}$ and $q_{\ell-1}^{(n)} \leq 2q_{\ell}^{(n)}$. Since \eqref{eq:(2)} does not hold, we conclude
\begin{equation}
\label{eq:(3)}
\forall n_0,\ell_0\in\mathbb{N},\quad \exists n\geq n_0,\; \exists \ell\geq \ell_0\,:\quad A_{n,l}.
\end{equation}
Then we can iteratively define a sequence $(n_k,\ell_k)$ with $n_k\to\infty$ and $\ell_k\to\infty$ such that $A_{n_k,\ell_k}$ is true for each $k\in\mathbb{N}$. More precisely, let $(n_1,\ell_1)$ be such that $1\leq \ell_1$, $1\leq n_1$ and $A_{\ell_1,n_1}$ is true, which is possible by \eqref{eq:(3)}. Suppose now we have $(n_k,\ell_k)$ for $1\leq k\leq m$ such that $A_{n_k,\ell_k}$ is true. By \eqref{eq:(3)}, there is an $n>n_m$ and $\ell>\ell_m$ such that $A_{n,\ell}$ is true. Set $n_{m+1}:=n$ and $\ell_{m+1}:=\ell$. By construction $n_k\to\infty$ and $\ell_k\to\infty$ and $A_{n_k,\ell_k}$ is true for all $k\in\mathbb{N}$. This contradicts \eqref{eq:(1)}. Thus, we have proven that \eqref{eq:(1)} and \eqref{eq:(2)} are equivalent. 

\medskip

By our assumption and the previous considerations \eqref{eq:(2)} holds. Let $\ell_0,n_0\in\mathbb{N}$ be chosen according to \eqref{eq:(2)}.
Furthermore, fix $n\geq n_0$, and $\ell\in L_n$ such that $\ell\geq \ell_0$ (exists since $L_n$ is infinite). Then \eqref{eq:(2)} leads to $q_{\ell-1}^{(n)}>2q_\ell^{(n)}$. Thus,
$$
q_{\ell-1}^{(n)}>2q_\ell^{(n)}\geq \max_{i\geq \ell}q_i^{(n)}
$$
follows implying that also $\ell-1\in L_n$. Since $\sharp L_n=\infty$, the latter considerations imply that $[\ell_0,\infty)\cap\mathbb{N}\subseteq L_n$. Thus, $q_{\ell-1}^{(n)}>2q_\ell^{(n)}$ holds for all $\ell\geq \ell_0$ by \eqref{eq:(2)}. Altogether, we derive that 
$$
q_\ell^{(n)}< \frac{1}{2}q_{\ell-1}^{(n)}<\frac{1}{2^2} q_{\ell-2}^{(n)} < \ldots < \frac{1}{2^{\ell-\ell_0}}q_{\ell_0}^{(n)} \leq \frac{1}{2^{\ell-l_0}} \|\psi_n\|_\infty
$$
holds for all $n\geq n_0$ and $\ell\geq \ell_0$. Hence, for $v\in\mathcal{V}(G)$ with $|v|=r\geq \ell_0 k$, the estimates
$$
|\psi_n(v)| \leq q_{\lfloor r/k\rfloor}^{(n)} \leq \frac{1}{2^{\lfloor r/k\rfloor-\ell_0}} \|\psi_n\|_\infty \leq \|\psi_n\|_\infty 2^{\ell_0} e^{-\frac{ln(2)}{k} r}
	\,,\qquad n\geq n_0\,,
$$
are deduced where $\lfloor r/k\rfloor$ denotes the largest integer $j$ satisfying $j\leq \frac{r}{k}$. Recall that $\mathcal{S}_r(v_0)$ denotes the sphere of radius $r$ and center $v_0$ in $G$. Since $(G,v_0)$ has sub-exponential growth rate, there is a constant $C_k>0$ such that $\sharp\mathcal{S}_r(v_0)<C_k \gamma_k^r$ for all $r\in\mathbb{N}$ where $\gamma_k:=e^{\frac{ln(2)}{2k}}>1$, see Definition~\ref{def:subexp}. Then a short computation yields
$$
\sum_{v\in V(G)\,:\, |v|\geq \ell_0k} |\psi_n(v)|^2 
	\leq \|\psi_n\|_\infty 2^{\ell_0} \sum_{r\geq \ell_0k} \sharp\mathcal{S}_r(v_0) e^{-\frac{ln(2)}{k} r}
	\leq \|\psi_n\|_\infty 2^{\ell_0} C_k \sum_{r\geq \ell_0k} \gamma_k^{-r}.
$$
The latter sum is convergent by the root test as $\gamma_k^{-1}<1$. Furthermore, the cardinality of all vertices $v\in V(G)$ with $|v|<\ell_0k$ is finite as $G$ is $d$-bounded graph. Hence,
$$
\|\psi_n\|_2^2 = 
\sum_{v\in V(G)} |\psi_n(v)|^2 = \sum_{v\in V(G)\,:\, |v|< \ell_0k} |\psi_n(v)|^2 + \sum_{v\in V(G)\,:\, |v|\geq \ell_0k} |\psi_n(v)|^2 
	\leq \widetilde{C}_k \|\psi_n\|_\infty
$$
follows for a suitable constant $\widetilde{C}_k>0$. Since $\lim_{n\to\infty}\|\psi_n\|_\infty=0$, we derive $\lim_{n\to\infty}\|\psi_n\|_2^2=0$, a contradiction as $\|\psi_n\|_2=1$. Thus, Lemma~\ref{lem:lastlem} is proven.
\end{proof}

\subsection{Proof of Proposition~\ref{prop:strictinclusion}}

Next we show that the inclusion of Theorem~\ref{thm:sigmainfty} can also be strict for graphs with sub-exponential growth rate. This is done by providing a specific example motivated by the considerations made in Section~\ref{sec:MainRes} about the $d$-regular tree. This construction is inspired by an example given in \cite{E}. In order to do so a more general construction of so called chain graphs is introduced next. 

\begin{definition}
\label{def:chaingraph}
Let $(G_k)_{k\in\mathbb{N}}$ be a sequence of finite graphs, $v^1_{k},v^2_{k}\in \mathcal{V}(G_k)$ for each $k\in\mathbb{N}$ and $\{k_\ell\}_{\ell\in\mathbb{N}} \subseteq \mathbb{N}$ be an increasing sequence. Then the corresponding {\it chain graph} is defined as follows:
\begin{itemize}
\item Begin with the graph $\mathbb{N}$, namely $V=\mathbb{N}$ and $E=\{(n,n+1)\,|\,n\in\mathbb{N}\}$.
\item For each $k\in \{k_\ell\}_{\ell=1}^\infty$, replace the vertex $k$ with the graph $G_k$.
\item The edges $(k-1,k)$, $(k,k+1)$ are replaced with the edges $\left(k-1,v^1_{k}\right)$, $\left(v^2_{k},k+1\right)$.
\item In case that $k=k_{\ell}=k_{\ell-1}+1$ the edge $(k-1,k)$ is replaced with the edge $\left(v^2_{k-1},v^1_{k}\right)$.
\end{itemize}
\end{definition}

%%%%%%%%%%%%%%%%
\begin{proof}[Proof of Proposition~\ref{prop:strictinclusion}]
Fix $d>2$ and let $T_d$ be the $d$-regular tree with root $v_0$, and denote by $G_k$ the finite subgraph $B_k(v_0)$ in $T_d$.
Let $\{k_\ell\}_{\ell=1}^\infty$ be the sequence defined by $k_\ell:=d^{\ell+1}$. For each $\ell\in\mathbb{N}$, let $v_\ell^{1},v_\ell^{2}$ be two vertices in $G_\ell$ of maximal distance from each other. Then let $G^d$ be the corresponding chain graph, see Definition~\ref{def:chaingraph} and a sketch in Figure~\ref{figure:ChainGraph}.

\begin{figure}[ht!]
\centering
\includegraphics[width=130mm]{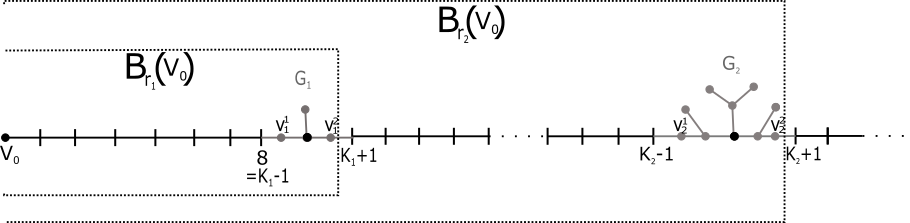}\\
\caption{The graph $G^d$ for $d=3$. 
\label{figure:ChainGraph}}
\end{figure}

Let $A_{G^d}$ and $A_{T_d}$ be the adjacency operator on $G^d$ and $T_d$ respectively, namely it is the Schr\"odinger operator of the form \eqref{eq:H} with $W(v):=deg(v)$. 
Next, we show that $G^d$ with root $v_0=1$ is of sub-exponential growth.

\medskip

A short computation gives that $\sharp G_k = 1+d\frac{(d-1)^{k}-1}{d-2} <\frac{d(d-1)^k}{d-2}$. Furthermore, $\delta_\ell :=\text{dist}(v_{\ell-1}^2,v_\ell^1) = d^{\ell+1}-d^\ell = (d-1)d^\ell$ holds implying $\sharp G_k< \delta_\ell$. Define $r_\ell := \text{dist}(v_0,v_\ell^2)$. Then $\delta_\ell < r_\ell-r_{\ell-1}$ holds. Next, we will show $\sharp B_{r_\ell}(v_0) < 2r_\ell, \ell\in\mathbb{N},$ via induction. For the base case, we have
$$
\sharp B_{r_1}(v_0) = k_1-1 + \sharp G_1 = d^2 + d
	< 2d^2+4 = 2(k_1 + 1) = 2r_1.
$$
The induction step is deduced by the estimate
$$
\sharp B_{r_\ell}(v_0) = \sharp B_{r_{\ell-1}}(v_0) + \delta_\ell + \sharp G_\ell < \sharp B_{r_{\ell-1}}(v_0) + 2\delta_\ell < 2 r_\ell
$$
invoking the induction hypothesis and the previous considerations. With this at hand, we derive by the definition of $G^d$ that
$$
\sharp B_{r_\ell+j}(v_0) = \sharp B_{r_\ell}(v_0) + j < 2(r_\ell + j),\qquad 0\leq j\leq k_{\ell+1}-k_\ell + 1.
$$
On the other hand decreasing the radius of $B_{r_\ell}(v_0)$ reduces the number of vertices in the ball by at least two in each step since $d>2$. Thus,
$$
\sharp B_{r_\ell-j}(v_0) < \sharp B_{r_\ell}(v_0) - \sum_{n=1}^j 2 < 2 (r_{\ell} - j),\qquad 1\leq j \leq 2\ell-1,
$$
follows. 

\medskip

Putting all together, we derive $\sharp B_k(v_0)<2k$ for each $k\in\mathbb{N}$. Furthermore, $G^d$ is a sub-graph of $T_d$. Thus, \cite[Corollary 4.5]{Mohar82} together with \cite[Section 7.c]{MW89} lead to
\[
\sigma_{\text ess}(A_{G^d})\subseteq \sigma(A_{G^d})\subseteq \left[-2\sqrt{d-1},2\sqrt{d-1}\right]=\sigma(A_{T_d}).
\]
On the other hand, one of the $\mathcal{R}$-limits of $A_{G^d}$ is the adjacency operator on the $d$-regular tree, for which $[-d,d]\subseteq\sigma_\infty(A_{T_d})$, see \cite[Theorem~1.1]{Brooks}. As a consequence, we conclude
\[
\emptyset \neq \Big[-d,d\Big]\backslash\left[-2\sqrt{d-1},2\sqrt{d-1}\right]
	\subseteq \bigcup_{H'\in\mathcal{R}(A_{G^d})} \sigma_{\infty}\left(H'\right) \backslash\ \sigma_{\text ess}(A_{G^d})
\]
since $d>2$.
\end{proof}
%%%%%%%%%%%%%%%%

% ------------------- Section III.2 ---------------------------
\section{Proof of Theorem~\ref{thm:fullresultsubexp}}
\label{sec:Thm-fullresultsubexp}

The proof follows from the results mentioned above, and the following theorem from \cite{BDE18}.

\begin{theorem}[{\cite[Theorem~2]{BDE18}}] 
\label{thm:rlimsubsetess}
Let $G$ be an infinite, connected $d$-bounded graph and $H$ a Schr\"odinger operator on $G$, then
\[
\bigcup_{H'\in\mathcal{R}(H)}\sigma(H')\subseteq\sigma_{\text ess}(H).
\]
\end{theorem}

With this at hand, we can prove Theorem~\ref{thm:fullresultsubexp}.

\begin{proof}[Proof of Theorem~\ref{thm:fullresultsubexp}]
%\textcolor{red}{check the proof again} V4.1
We already know from Theorem~\ref{thm:rlimsubsetess} and Theorem~\ref{thm:sigmainfty} that
\[
\bigcup_{H'\in\mathcal{R}(H)}\sigma(H')\subseteq\sigma_{\text ess}(H)\subseteq\bigcup_{H'\in\mathcal{R}(H)}\sigma_\infty(H').
\]
Thus, it suffices to prove for each $H'\in\mathcal{R}(H)$, that
\[
\sigma_\infty(H')\subseteq\sigma(H').
\]
In order to do so, recall the notions introduced in Section~\ref{ssect:Shnol}. Since $G$ admits a uniform sub-exponential growth also any $\mathcal{R}$-limit $(H',G',v'_0)$ of $(H,G,v_0)$ is of uniform sub-exponential growth. Indeed, for each $\gamma>1$ there exists a constant $C>0$ (independent of the root) such that $\sharp \mathcal{S}_r(u)<C\gamma^r$ for all $u\in V(G)$ and $r\in\mathbb{N}$, see Definition~\ref{def:subexp}. Let $\gamma>1$ be arbitrary and $r\in\mathbb{N}$. Since $(H',G',v'_0)$ is an $\mathcal{R}$-limit of $(H,G,v_0)$, there is a vertex $u\in V(G)$ such that the subgraphs $B_r(v'_0)$ and $B_r(u)$ are isomorphic (see Definition~\ref{def:Hgraphsconv}). In particular, the spheres contained in these balls have the same cardinality, namely $\sharp \mathcal{S}_r(u)=\sharp \mathcal{S}_r(v'_0)$. Since $\sharp \mathcal{S}_r(u)< C\gamma^r$ where the corresponding constant $C$ is independent of $u\in V(G)$, we derive $\sharp \mathcal{S}_r(v'_0)<C\gamma^r$. Thus, the map $V(G')\ni v'\mapsto e^{-\alpha\ \text{dist}(v',v'_0)}$ is an element $\ell^2(G')$ for any $\alpha>0$. 

\medskip

Let $\lambda\in\sigma_\infty(H')$, then by definition there is a bounded generalized eigenfunction $\varphi$ of $H'$ with eigenvalue $\lambda$. By the previous considerations
$$
V(G')\ni v'\mapsto e^{-\alpha\ \text{dist}(v',v'_0)}\varphi(v')
$$
is an element of $\ell^2(G')$ as $\varphi$ is uniformly bounded. Thus, $\varphi$ is sub-exponentially bounded implying $\lambda\in\sigma(H')$ by Proposition~\ref{prop:HK4.8}.
\end{proof}

\appendix
\section{Sparse trees with sparse cycles}
We provide here the proofs of Lemma~\ref{lem:CGspectra} and Lemma~\ref{lem:SGspectra} that were needed to compute the essential spectrum of the graphs named sparse trees with sparse cycles, see Section~\ref{sec:example2}.

\begin{proof}[Proof of Lemma~\ref{lem:CGspectra}]
The spectrum is calculated by using the periodicity of the underlying graph of $A_{\mathcal{C}}$ (following e.g.\ \cite[Chapter 5]{SimSz11} and \cite[Chapter XIII.16]{ReedSimon4}). The operator $A_{\mathcal{C}}$ is unitary equivalent to a direct integral of the operators
\[
A^{(\theta)}=A_{\mathbb Z}+2\cos\theta \delta_0.
\]
For indeed, define $\mathcal{F}:\ell^2(\mathcal{C})\to L^2\left(\mathbb{T},\frac{d\theta}{2\pi}; \ell^2(\mathbb{Z})\right)$ by
\[
\left(\mathcal{F}\psi\right)(\theta,\ell) :=\sum_{k=-\infty}^{\infty} \psi(k,\ell) e^{-ik\theta}.
\]
This operator is first defined on the finitely supported functions on $\mathcal{C}$ and then extended to $\ell^2(\mathcal{C})$ using
\[
\sum_{\ell\in\mathbb{Z}}\int_{\mathbb{T}}\left\Vert\mathcal{F}\psi(\cdot,\ell)\right\Vert^2\frac{d\theta}{2\pi} = \sum_{k,\ell\in\mathbb{Z}}\left|\psi(k,\ell)\right|^2.
\]
The inverse $\mathcal{F}^{-1}:L^2\left(\mathbb{T},\frac{d\theta}{2\pi}; \ell^2(\mathbb{Z})\right)\to \ell^2(\mathcal{C})$ of $\mathcal{F}$ is defined by
\[
(\mathcal{F}^{-1}f)(k,\ell)=\int e^{i k\theta}f(\theta,\ell)\frac{d\theta}{2\pi}.
\]
Then $\mathcal{F}$ is a unitary transformation and 
\[\left[(\mathcal{F}A\mathcal{F}^{-1})f\right](\theta,\ell)=\left(A^{(\theta)}f\right)(\theta,\ell)
\]
holds which can be checked by a short computation. 
Thus, in order to compute the spectrum of $A_{\mathcal{C}}$ we need to compute the spectrum of the direct integral. Define $R(z):=\left(A^{(\theta)}-z\right)^{-1}$ and $R_0(z):=\left(A_{\mathbb Z}-z\right)^{-1}$.
By the Aronszajn-Krein formula (see e.g.\ \cite{SimRanOne93}) we get
\[
m(z)=\langle\delta_0,R(z)\delta_0\rangle= \frac{m_0(z)}{1+2\cos\theta m_0(z)},
\]
where
\[
m_0(z)=\langle\delta_0,R_0(z)\delta_0\rangle=-\frac{1}{\sqrt{z^2-4}}
\]
is the Borel transform corresponding to the adjacency operator on $\mathbb{Z}$, see e.g.\ \cite{SimSz11}.
We know that $\sigma(A_{\mathbb Z})=\sigma_{\textrm{ess}}(A_{\mathbb Z})=[-2,2]$ and therefore $[-2,2]\subseteq \sigma\left(A^{(\theta)}\right)$ for all $\theta$. We can get additional points in the spectrum of $A^{(\theta)}$ if  $1+2\cos\theta m_0(z)$ vanishes. Thus, 
$\sqrt{z^2-4}=2\cos\theta$ follows implying
\[
z_{\pm}= \pm2\sqrt{1+\cos^2\theta}.
\]
The choice of the $\pm$-branch of the root is determined by requiring that $m(z)=-\nicefrac{1}{z}+O\left(\nicefrac{1}{z^2}\right)$, see \cite[pg.\ 53]{SimSz11}. As also follows by the fact that $A_{\mathcal{C}}$ is symmetric, both points $z_+$ and $z_-$ are included in the spectrum of $A^{(\theta)}$ for suitable choices of $\theta$.
%Both $z_+$ and $z_-$ are chosen for a suitable choice of $\theta$. This can be also deduced from the fact that $A_{\mathcal{C}}$ is symmetric.
Notice that the vector $\delta_0$ is not cyclic for $A^{(\theta)}$. However, since the difference $A^{(\theta)}-A_{\mathbb Z}=2\cos\theta(\delta_0,\cdot)\delta_0$ is of rank, we get one additional point in the spectrum for each $\theta$. This is the one we just have found before.
Using continuity in $\theta$ to construct approximate eigenfunctions and integrating over $\theta$ we get that
$\sigma\left(A_{\mathcal{C}}\right)=[-2\sqrt{2},2\sqrt{2}]$.
%where, since the spectrum of $A_{\mathcal{C}}$ is symmetric both the points $z_{\pm}$ are in the spectrum.
\end{proof}

\medskip

\begin{proof}[Proof of Lemma~\ref{lem:SGspectra}]
We shall calculate the spectrum of the adjacency operator on the graph $S_k$.
This graph is spherically homogeneous, and thus the operator $A_{S_k}$ is unitary equivalent to a direct sum of one dimensional Jacobi operators.
Indeed, according to Theorem~2.4 of \cite{Bre07}, 
\[
A_{S_k}\cong J_0\oplus A_{\mathbb N}\oplus\ldots\oplus A_{\mathbb N},
\]
where $J_0$ is a Jacobi matrix with parameters 
\[
a_n=\begin{cases}
\sqrt{k} & n=1 \\
1 & n>1,
\end{cases}
\]
\[b_n\equiv0,\]
and $A_{\mathbb N}$ appears in the direct sum $(k-1)$-times.
Since $J_0$ is a finite rank perturbation of $A_{\mathbb N}$ we have that $\sigma_{\text ess}(J_0)=\sigma_{\text ess}(A_{\mathbb N})=\sigma(A_{\mathbb N})=[-2,2]$.
Thus we should only calculate the discrete spectrum of $J_0$. By coefficient stripping (see e.g. \cite[Theorem~3.2.4]{SimSz11}) we get the following relation for the $m$-function of $J_0$
\[
m(z)=\frac{-1}{z+k m_{\mathbb N}(z)},
	\quad\text{ where }\quad
m(z)=\langle\delta_1,(J_0-z)^{-1}\delta_1\rangle.
\]
Notice that $\delta_1$ is a cyclic vector for $J_0$, thus additional points in the spectrum of $J_0$ exist only if
\begin{equation} \label{eq:mISksol}
z+k m_{\mathbb N}(z)=0.
\end{equation}
Using the known expression (see, e.g.\ \cite{SimSz11})
\begin{equation*}
m_{\mathbb{N}}\left(z\right)=\frac{-z+\sqrt{z^{2}-4}}{2}
\end{equation*}
we get that the solutions of \eqref{eq:mISksol} are $z_{\pm}=\pm\frac{k}{\sqrt{k-1}}$.
Thus $z_{\pm}\in\sigma\left(A_{S_k}\right)$, and
\[
\sigma\left(A_{S_k}\right)=\left\{-\frac{k}{\sqrt{k-1}}\right\}\cup[-2,2]\cup\left\{\frac{k}{\sqrt{k-1}}\right\}.
\]
\end{proof}

\end{document}